\documentclass{amsart}
\usepackage{graphicx}
\usepackage{amssymb}
\usepackage{amsmath}
\usepackage{amsthm}
\usepackage{amsaddr}
\usepackage{enumerate}
\usepackage{bbm}
\usepackage{caption}
\usepackage{subcaption}
\usepackage{calligra}
\usepackage{MnSymbol}
\usepackage{bbm}
\usepackage{yfonts}
\usepackage{upgreek}
\usepackage{amssymb} %blacksquare
\usepackage{adforn}
\usepackage{pgfornament}
\usepackage{enumerate}
\usepackage{verbatim}
\usepackage{collectbox}
\usepackage{xcolor}
\usepackage{MnSymbol}
\usetikzlibrary{decorations.pathreplacing,calc,intersections}

\DeclareMathAlphabet{\mathcalligra}{T1}{calligra}{m}{n}

\newtheorem{thm}{Theorem}[section]
\newtheorem{cor}[thm]{Corollary}

\newtheorem{prop}[thm]{Proposition}
\theoremstyle{definition}
\newtheorem{definition}[thm]{Definition}
\newtheorem{remark}[thm]{Remark}

\newtheorem*{defn*}{Definition}
\newtheorem*{rems*}{Remarks}
\newtheorem*{rem*}{Remark}

\usepackage[pagewise]{lineno}%\linenumbers

\numberwithin{equation}{section}

\begin{document}

\title[Polygons with Parallel Opposite Sides] {
Convex Polygons with Parallel Opposite Sides: Convergence, Reconstruction, and Isoperimetric Inequalities}

\author[I. Konicer, B. Murawski, B. Rogala, T. Wieczorek, and M. Zwierzy\'nski]{Izabella Konicer$^1$, Bart\l{}omiej Murawski$^2$, Bruno Rogala$^3$, Tomasz Wieczorek$^4$, and Micha\l{} Zwierzy\'nski$^{5, \pentagram}$}
\email{$^{1}$Izabella.Konicer.stud@pw.edu.pl, ORCID: 0009-0007-1046-3761}
\email{$^{2}$B.H.Murawski@gmail.com}
\email{$^{3}$Bruno.Rogala.stud@pw.edu.pl, ORCID: 0009-0009-8914-4644}
\email{$^{4}$Tomasz.Wieczorek.stud@pw.edu.pl, ORCID: 0009-0001-5682-4602}
\email{$^{5}$Michal.Zwierzynski@pw.edu.pl}
\email{ORCID: 0000-0002-9627-1563}

\address{Warsaw University of Technology\\
Faculty of Mathematics and Information Science\\
ul. Koszykowa 75\\
00-662 Warsaw, Poland}

\thanks{$^{\pentagram}$Corresponding author}



%%%%%%%%%%%%%%%%%%%%%%%%%%%%%%%%%%%%
%%%%%%%%% COMMANDS %%%%%%%%%%%%%%%%%
%%%%%%%%%%%%%%%%%%%%%%%%%%%%%%%%%%%%

\newcommand{\Eq}{{\mathrm{E}}}
\newcommand{\E}{\mathbb{E}}
\newcommand{\M}{{M}}
\newcommand{\N}{{N}}
\newcommand{\Css}{{\mathrm{CSS}}}
\newcommand{\Gcs}{{\mathrm{GCS}}}
\newcommand{\C}{{C}}
\newcommand{\D}{{D}}
\newcommand{\ppos}{\mathrm{PPOS}}
\newcommand{\cppos}{\mathrm{CPPOS}}
\newcommand{\F}{\text{F}}
\newcommand{\EF}{\mathcal{E}}
\newcommand{\Cwms}{\mathrm{CWMS}}
\newcommand{\Sms}{\mathrm{SMS}}
\newcommand{\lcm}{\mathrm{lcm}}

%%%%%%%%%%%%%%%%%%%%%%%%%%%%%%%%%%%%%%%%%%%%%%
%%%%%%%%%%%%%%%%%%%%%%%%%%%%%%%%%%%%%%%%%%%%%%
%%%%%%%%% BEGIN %%%%%%%%%%%%%%%%%%%%%%%%%%%%%%
%%%%%%%%%%%%%%%%%%%%%%%%%%%%%%%%%%%%%%%%%%%%%%

\begin{abstract}
In this paper, we study affine $\lambda$-equidistants of convex polygons with parallel opposite sides ($\cppos$es), introduced by M. Craizer et al. \linebreak (\emph{Polygons with Parallel Opposite Sides}, Discrete \& Computational Geometry \textbf{50} (2013), 474--490), and relate them to affine $\lambda$-equidistants of smooth convex curves. We show how, and under which conditions, the original $\cppos$ can be reconstructed from its Wigner caustic and centre symmetry set.

We prove a discrete version of the improved isoperimetric inequality:
\begin{align*}
L(\mathcal{P})^2\geqslant 8n\tan\left(\frac{\pi}{2n}\right)\cdot\big(
A(\mathcal P)+2\left|A^\ast\left(\Eq_{0.5}(\mathcal P)\right)\right|\big),
\end{align*}
where $\mathcal{P}$ is a $\cppos$ with $2n$ vertices, $L(\mathcal{P})$ is its perimeter, $A(\mathcal{P})$ is the area enclosed by $\mathcal{P}$, and $A^\ast\left(\Eq_{0.5}(\mathcal P)\right)$ denotes the oriented area of the Wigner caustic of $\mathcal{P}$. Moreover, equality holds if and only if $\mathcal{P}$ is an equiangular $\cppos$ of constant width.

We also prove sharp area estimates for the oriented areas of the Wigner caustic and the centre symmetry set of a $\cppos$. More precisely, we show that the absolute value of the oriented area of the Wigner caustic is at most one quarter of the area of $\mathcal{P}$, while the absolute value of the oriented area of the centre symmetry set is at most the area of $\mathcal{P}$. Both bounds are sharp.
\end{abstract}

\maketitle

\section*{Statements and Declarations}
Competing Interests: The authors declare that they have no competing interests.

Funding: The authors did not receive support from any organization for the submitted work.

%%%%%%%%%%%%%%%%%%%%%%%%
%%%%% Introduction %%%%%%%%%%%%%
%%%%%%%%%%%%%%%%%%%%%%%%

\section{Introduction}

\noindent Discrete differential geometry is based on the idea that a discrete object should not merely approximate a smooth one, but should preserve its essential geometric structure. From this point of view, polygonal models are most useful when the fundamental constructions of smooth differential geometry admit natural, intrinsic, and computable discrete counterparts. This philosophy is particularly fruitful in affine differential geometry, where many classical objects are defined not by Euclidean metric quantities, but by affine-invariant constructions such as parallel tangents, chords, envelopes, and oriented areas (see \cite{Bobenko, bogosel, CraizerDisc, CraizerPPOS, kwong}, and the literature therein).

The aim of the present paper is to develop this point of view for affine \linebreak$\lambda$-equidistants of convex polygons with parallel opposite sides. These polygons, introduced in \cite{CraizerPPOS}, form a natural discrete analogue of smooth convex curves. Indeed, they may be obtained by sampling tangent lines to a convex curve in pairs of parallel directions, and many affine constructions known in the smooth case have direct polygonal counterparts. In this sense, $\cppos$es are not only combinatorial objects, but also a discrete affine-differential model of convex ovals.

For a smooth closed planar curve, affine $\lambda$-equidistants are defined by taking affine combinations
$\lambda a+(1-\lambda)b$
of pairs of points whose tangent lines are parallel. The special case $\lambda=0.5$ is the Wigner caustic, also known in some contexts as the area evolute and as the middle hedgehog. Closely related to it is the centre symmetry set, which may be regarded as an affine analogue of the centre of symmetry: for a~centrally symmetric oval, this set degenerates to a single point, while in general it becomes a singular curve measuring the failure of central symmetry. These objects have been studied from several complementary perspectives, including singularity theory, affine geometry, convex geometry, and isoperimetric inequalities; see for instance \cite{Berry, bogosel, CraizerDR2, Art, DFJSymmetryDefect, DomitrzRios, DomitrzRiosRuas, DJRR, DomitrzMR, D-Z3, GB, ZD_Wigner, GiblinHoltom, GiblinReeve2, GiblinZ2, Janeczko, JaneczkoJelonekRus, Dominika, S6, Zwierz1, ZwierzMix}, and the literature therein.

The polygonal case is interesting for at least two reasons. First, it provides a discrete model in which the constructions are elementary and explicitly computable: affine equidistants, the Wigner caustic, and the centre symmetry set are again polygonal objects determined by the combinatorics and affine geometry of the original polygon. Secondly, the discrete theory reflects the smooth one surprisingly well. Thus, one can prove certain results directly in the polygonal setting and then interpret them as discrete analogues, or even as approximating versions, of the corresponding theorems for smooth convex curves. This gives a useful bridge between affine differential geometry and discrete geometry.

In this paper we study this bridge in detail. We first recall and develop the basic geometry of $\cppos$es and their affine $\lambda$-equidistants. We show that the construction of polygonal approximations by tangent lines is compatible with taking affine equidistants: under natural assumptions, taking the $\lambda$-equidistant of a polygonal approximation gives the same polygonal object as approximating the $\lambda$-equidistant of the original oval. We also prove a Hausdorff convergence result, showing that, as the set of chosen directions becomes finer, the corresponding polygonal chains converge to the smooth affine equidistant.

We then analyze the singularities of polygonal equidistants. In the discrete setting, cusps are described by a simple condition involving consecutive edges, or equivalently, by the way the affine parameters of consecutive great diagonals cross the level $\lambda$. This leads to a parity result: for a generic $\lambda\neq 0.5$, the number of cusps of $\Eq_\lambda(\mathcal{P})$ is even, whereas the Wigner caustic $\Eq_{0.5}(\mathcal{P})$ has an odd number of cusps. This mirrors the behaviour known from the smooth theory and emphasizes that $\cppos$es retain the essential affine geometry of parallel pairs.

Another part of the paper is devoted to reconstruction. We investigate when a~$\cppos$ can be recovered from its Wigner caustic and centre symmetry set. Since these two singular sets encode information about the great diagonals and about the midpoints of opposite vertices, they carry a substantial amount of data about the original polygon. We present a reconstruction procedure and formulate conditions under which the reconstruction is possible. This shows that, in the polygonal affine setting, the Wigner caustic and the centre symmetry set are not merely auxiliary objects, but can be used as geometric fingerprints of the original polygon.

Finally, we prove isoperimetric-type results. Our main inequality is a discrete version of the improved isoperimetric inequality involving the oriented area of the Wigner caustic:
\begin{align*}
L(\mathcal{P})^2\geqslant
8n\tan\left(\frac{\pi}{2n}\right)
\cdot\big(
A(\mathcal{P})+2\left|A^\ast\left(E_{0.5}(\mathcal{P})\right)\right|
\big),
\end{align*}
where $\mathcal{P}$ is a $\cppos$ with $2n$ vertices. Equality holds precisely for equiangular $\cppos$es of constant width. We also obtain sharp estimates for the oriented areas of the Wigner caustic and the centre symmetry set in terms of the area enclosed by the initial polygon. These estimates provide discrete counterparts of known smooth phenomena and show that the singular affine sets associated with a $\cppos$ control a significant part of its global geometry.

The paper is organized as follows. In Section~2 we introduce $\cppos$es, their $\lambda$-equidistants, including the Wigner caustic, and the centre symmetry set, and we discuss the relation between the discrete and smooth constructions. Section~3 is devoted to reconstruction from the Wigner caustic and the centre symmetry set. In Section~4 we prove the discrete improved isoperimetric inequality and the sharp oriented area estimates.

%%%%%%%%%%%%%%%%%%%%%%%%%%%%%%%%%%%%%%%%%%%%%%%%%%%%%
%%%%%%%%%%%%%%%%%%%%%%%%%%%%%%%%%%%%%%%%%%%%%%%%%%%%%%
%%%%%%%%%%%%%%%%%%%%%%%%%%%%%%%%%%%%%%%%%%%%%%%%%%%%%%
\section{Convex Polygons with Parallel Opposite Sides}

\noindent We begin by introducing the class of convex polygons that constitutes the main object of this paper. In this section, we fix the notation, present the basic definitions, and describe the geometric structure of convex polygons with parallel opposite sides. These polygons provide a discrete counterpart of smooth convex curves and form the natural setting for the study of affine equidistants, the Wigner caustic, and the centre symmetry set. We also describe a construction of such polygons from a prescribed family of directions.

Throughout this paper, the vertices and edges of a polygon will be indexed periodically. More precisely, we write
$P_{i+2n}=P_i$ and 
$e_{i+2n}=e_i$
for every $i\in\mathbb{Z}$. Unless stated otherwise, all definitions, conditions, and identities involving the index $i$ are understood to hold for every $i\in\mathbb{Z}$.

\begin{definition}
\label{CPPOS definition}
Let $\mathcal{P}$ be a convex planar $2n$-gon for some integer $n\geqslant 3$. Denote its vertices by $P_i\in\mathbb{R}^2$, $i\in\mathbb{Z}$, and its edges as $e_i:=P_{i+1}-P_i$. 
If $\mathcal{P}$ satisfies the conditions:
\begin{enumerate}[(1)]
\item $e_i\|e_{i+n}$,
\item $\left[e_i,e_{i+1}\right]>0$,
\item $\left<e_i,e_{i+1}\right> >0$
\end{enumerate}
for all $i$, where $[\cdot,\cdot], \left<\cdot,\cdot\right>$ denote determinant product and scalar product, respectively, then we will call $\mathcal{P}$ a \textit{convex polygon with parallel opposite sides}, abbreviated as a $\cppos$.
\end{definition}

\begin{remark}
The definition used in \cite{CraizerPPOS} also permits $\cppos$es with acute interior angles. In the present paper, we restrict our attention to polygons whose interior angles are obtuse. Equivalently, the angle between each pair of consecutive oriented edge vectors $e_i$ and $e_{i+1}$ is acute, as expressed by condition~(3) in Definition \ref{CPPOS definition}.

This convention allows cusp singularities to be interpreted geometrically exactly as in Definition~\ref{Def:cusp}; see also Proposition~\ref{cusp_equivalence}. Namely, a cusp occurs when two consecutive oriented edges of the corresponding polygonal chain form an acute angle. Under the convention adopted in \cite{CraizerPPOS}, two edges meeting at an obtuse angle could also be classified as forming a cusp.
\end{remark}

Let
\[
d_i:=P_{i+n}-P_i
\]
denote the vector of the \(i\)-th great diagonal of \(\mathcal{P}\). Let \(D_i\) be the point of intersection of the two great diagonals \(P_iP_{i+n}\) and \(P_{i+1}P_{i+n+1}\), and let \(\lambda_i\) be the signed affine parameter defined by
\[
D_i-P_i=\lambda_i d_i.
\]
Thus, \(\lambda_i\) measures the signed fraction of the great diagonal from the vertex \(P_i\) to the intersection point \(D_i\). By the intercept theorem, \(\lambda_i\) also satisfies
\[
D_i-P_{i+1}=\lambda_i d_{i+1}.
\]
Clearly, \(D_{i+n}=D_i\), and consequently \(\lambda_{i+n}=1-\lambda_i\). Moreover, the homothety with center \(D_i\) which maps the side \(e_i\) to the opposite side \(e_{i+n}\) gives
\begin{equation}\label{eq:homothety}
e_{i+n}
=
-\frac{\lambda_{i+n}}{\lambda_i}e_i
=
\frac{\lambda_i-1}{\lambda_i}e_i.
\end{equation}

The reader may notice that the authors of \cite{CraizerPPOS} use a different indexing convention for these objects, namely \(e(i+\frac{1}{2})\), \(D(i+\frac{1}{2})\), and \(\lambda(i+\frac{1}{2})\). Although this convention may be more intuitive, keeping it here would make some of our formulae rather long and difficult to read. Therefore, we replace it by the shorter notation \(e_i\), \(D_i\), and \(\lambda_i\), respectively.

\begin{definition}\label{def:convex-equidistant}
Let $\mathcal{P}$ be a $\cppos$. For $\lambda\in\mathbb{R}$, let 
$\Eq_{\lambda}(\mathcal{P})$ denote the polygon whose vertices, in cyclic order, are
\[
P_i(\lambda):=P_i+\lambda d_i.
\]
The polygon $\Eq_{\lambda}(\mathcal{P})$ is called the \textit{affine $\lambda$-equidistant of $\mathcal{P}$}, or simply the \textit{$\lambda$-equidistant of $\mathcal{P}$}. 
Its edges will be denoted by
\[
e_i(\lambda):=P_{i+1}(\lambda)-P_i(\lambda).
\]
In particular, $\Eq_{0.5}(\mathcal{P})$ is called the \textit{Wigner caustic of $\mathcal{P}$}.
\end{definition}

\begin{definition}\label{def:csscppos}
Let $\mathcal{P}$ be a $\cppos$. Let the \textit{centre symmetry set} of $\mathcal{P}$, which we will denote as $\Css(\mathcal{P})$, be the polygon whose vertices are $D_i$.  \end{definition}

Since $P_{i+n}(0.5)=P_i(0.5)$ and $D_{i+n}=D_i$, both the Wigner caustic and the centre symmetry set are understood as closed polygonal chains with vertices indexed by $i=1,\ldots,n$, and hence are traversed only once.

We illustrate the centre symmetry set and the Wigner caustic of a hexagonal $\cppos$ in Figure \ref{Fig:PCSSWCfigure}.

\begin{figure}[h]
    \centering
    \includegraphics[scale=0.44]{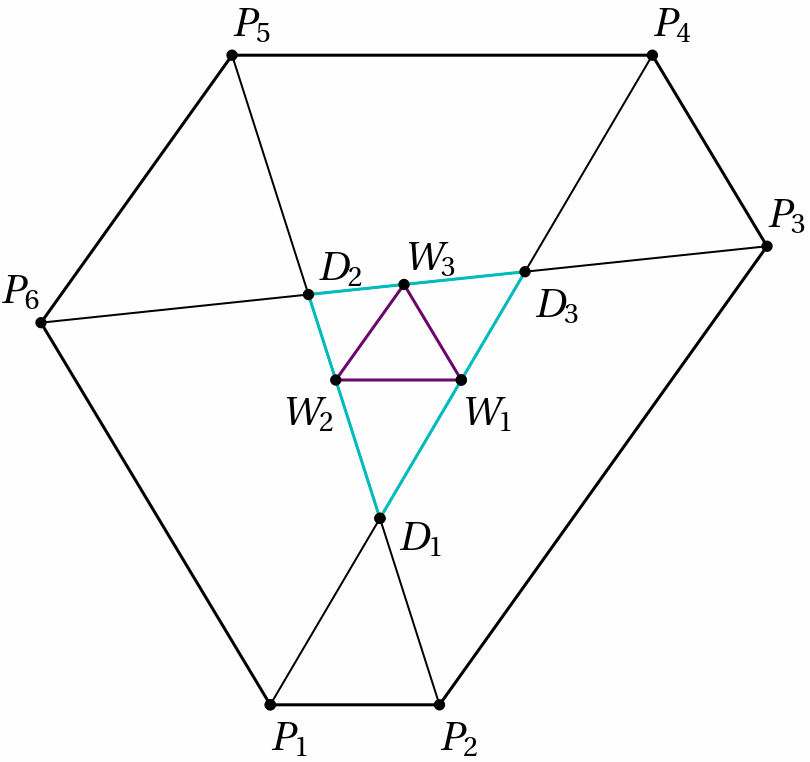}
    \caption{$\cppos$  $\mathcal{P}$ together with its Wigner caustic and centre symmetry set}
    \label{Fig:PCSSWCfigure}
\end{figure}

Before we proceed, let us define the smooth counterparts of the sets introduced above for polygons.

\begin{definition}
Let $M\subset\mathbb{R}^2$ be a smooth regular closed planar curve, that is, the image of a $C^{\infty}$ immersion $S^1\to\mathbb{R}^2$. Let $\lambda\in\mathbb{R}$. A pair of distinct points $a,b\in M$ is called a \textit{parallel pair} if the tangent lines to $M$ at $a$ and $b$ are parallel. The \textit{affine $\lambda$-equidistant} of $M$ is the set
\[
\Eq_{\lambda}(M):=
\big\{
\lambda a+(1-\lambda)b
\colon
\ a,b \text{ form a parallel pair of } M
\big\}.
\]
As in the case of $\cppos$es, the set $\Eq_{0.5}(M)$ is called the \textit{Wigner caustic} of $M$.
\end{definition}

Notice that the notation $\Eq_\lambda$ is now used in two contexts: for smooth curves and for $\cppos$es. Its meaning should always be clear from the argument.

The singular set now known as the Wigner caustic was introduced in the context of the semiclassical limit of Wigner's phase-space representation of a quantum state~\cite{Berry}. In this limit, the Wigner function associated with the classical correspondence $\mathcal C$ of a pure quantum state attains large values not only near $\mathcal{C}$, but also in a~neighbourhood of its Wigner caustic. Remarkably, the oriented area of the Wigner caustic of an oval yields a refinement of the classical isoperimetric inequality, as well as of several related isoperimetric inequalities \cite{Cufi, Zhang1, Zwierz1, Zwierz3, Zwierz2, ZwierzMix}. A broader construction, known as the middle hedgehog, provides a natural generalization of the Wigner caustic \cite{S6}. The Wigner caustic also appears in one of the two constructions of two-dimensional improper affine spheres; this construction extends to higher dimensions and is closely related to solutions of the Monge--Amp\`ere equation \cite{CraizerDR2, CraizerIas}. The singularities of the Wigner caustic and of affine equidistants are also widely studied in convex geometry, singularity theory and differential geometry, where they provide natural geometric models and play an important role in the classification and analysis of singular phenomena \cite{CraizerDR2, CraizerPPOS, Art, DFJSymmetryDefect, DomitrzRiosRuas, DJRR, DomitrzMR, D-Z3, GB, ZD_Wigner, GiblinReeve2, JaneczkoJelonekRus, Rochera1, SC1, Xie}.

\begin{remark}
In \cite{CraizerPPOS}, which provided the main inspiration for our study of $\cppos$es, the Wigner caustic is referred to as the Area Evolute.
\end{remark}

One can easily see that
\begin{equation} \label{equation1 krawedzie remark}
e_i(\lambda)=(1-\lambda)e_i+\lambda e_{i+n}
\end{equation}
and we can further derive
\begin{equation} \label{equation2 krawedzie remark}
e_i(\lambda) =\frac{\lambda_i-\lambda}{\lambda_i}e_i.
\end{equation}

\begin{definition}
Let $\mathcal{P}$ be a $\cppos$. A number $\lambda\in\mathbb{R}$ is called \textit{generic} for $\mathcal{P}$ if no edge of $\Eq_{\lambda}(\mathcal{P})$ degenerates to a point.
\end{definition}

\begin{remark}
A value $\lambda$ is non-generic precisely when two consecutive great diagonals intersect at a point that divides both of them in the affine ratio $\lambda:(1-\lambda)$. In this situation, the corresponding edge of the $\lambda$-equidistant degenerates to their point of intersection. Equivalently, $\lambda$ is non-generic if and only if $\lambda=\lambda_i$ for some~$i$. Since $\lambda_{i+n}=1-\lambda_i$, the value $\lambda$ is generic if and only if $1-\lambda$ is generic. Unless explicitly stated otherwise, all values of $\lambda$ considered in what follows are assumed to be generic.
\end{remark}

One possible construction of a $2n$-sided $\cppos$ proceeds as follows. We begin with a set $\nu\subset[0,\pi)$ consisting of $n$ distinct directions. Here, the direction of a~line is understood as the angle, considered modulo $\pi$, that the line makes with a~fixed reference direction. We then choose $2n$ lines whose directions are given by the elements of $\nu$, each occurring twice, and arrange them in increasing cyclic order of their directions. The vertices of the resulting polygon are defined as the intersection points of consecutive lines. In general, however, an arbitrary choice of the set $\nu$ and of the corresponding lines does not necessarily yield a $\cppos$.

\begin{definition}
Let
$\nu=\{\alpha_1,\ldots,\alpha_m\}\subset [0,\pi)$, where
$m\geqslant 3$
and 
$0\leqslant\alpha_1<\alpha_2<\cdots<\alpha_m<\pi.$
We say that $\nu$ is an \textit{allowable set of directions} if
$\alpha_i-\alpha_{i-1}\in\left(0,\frac{\pi}{2}\right)$
for
$i=2,\ldots,m,$
and
$\alpha_1+\pi-\alpha_m\in\left(0,\frac{\pi}{2}\right).$
\end{definition}

\begin{definition}
Let $M$ be a closed curve whose only singularities, if any, are cusp singularities. 
By a \textit{cusp} of a curve we mean a singular point at which the curve is locally diffeomorphic (in the source and in the target) to the standard cusp
$t\mapsto (t^2,t^3)$
at $t=0$. We assume that $M$ has no inflection points and that its rotation number is equal to either $1$ or $\frac{1}{2}$. Since $M$ has only cusp singularities, there exists a continuous normal vector field along $M$. In the case when the rotation number is $\frac{1}{2}$, this field is understood on the double cover of $M$.

Let $\nu$ be an allowable set of directions for $M$. We define $\ppos_\nu(M)$ to be the closed polygonal chain constructed as follows. Let $\gamma$ be a parametrization of $M$, and let
$s_1<s_2<\cdots<s_{2m}$
be the parameter values, listed in cyclic order, such that the tangent lines to $M$ at the points $\gamma(s_i)$ have directions belonging to $\nu$. Denote by~$T_i$ the tangent line to $M$ at $\gamma(s_i)$. The vertices of $\ppos_\nu(M)$ are defined by
\[
P_i:=T_i\cap T_{i+1},
\qquad i=1,\ldots,2m,
\]
where we use the cyclic convention $T_{2m+1}=T_1$. Thus, $\ppos_\nu(M)$ is the closed polygonal chain with consecutive vertices
\[
P_1,P_2,\ldots,P_{2m}.
\]
If the rotation number of $M$ is $\frac{1}{2}$, then, of course,
$P_i=P_{m+i}$
for $i=1,\ldots,m$.
\end{definition}

One should keep in mind that only the lines containing the sides of $\ppos_\nu(M)$ are required to be tangent to $M$. If $M$ has cusp singularities, two consecutive tangent lines may intersect before reaching their points of tangency -- see Figure~\ref{FigCPPOSNotReaching}.

\begin{figure}[h]
    \centering
    \includegraphics[width=0.3\linewidth]{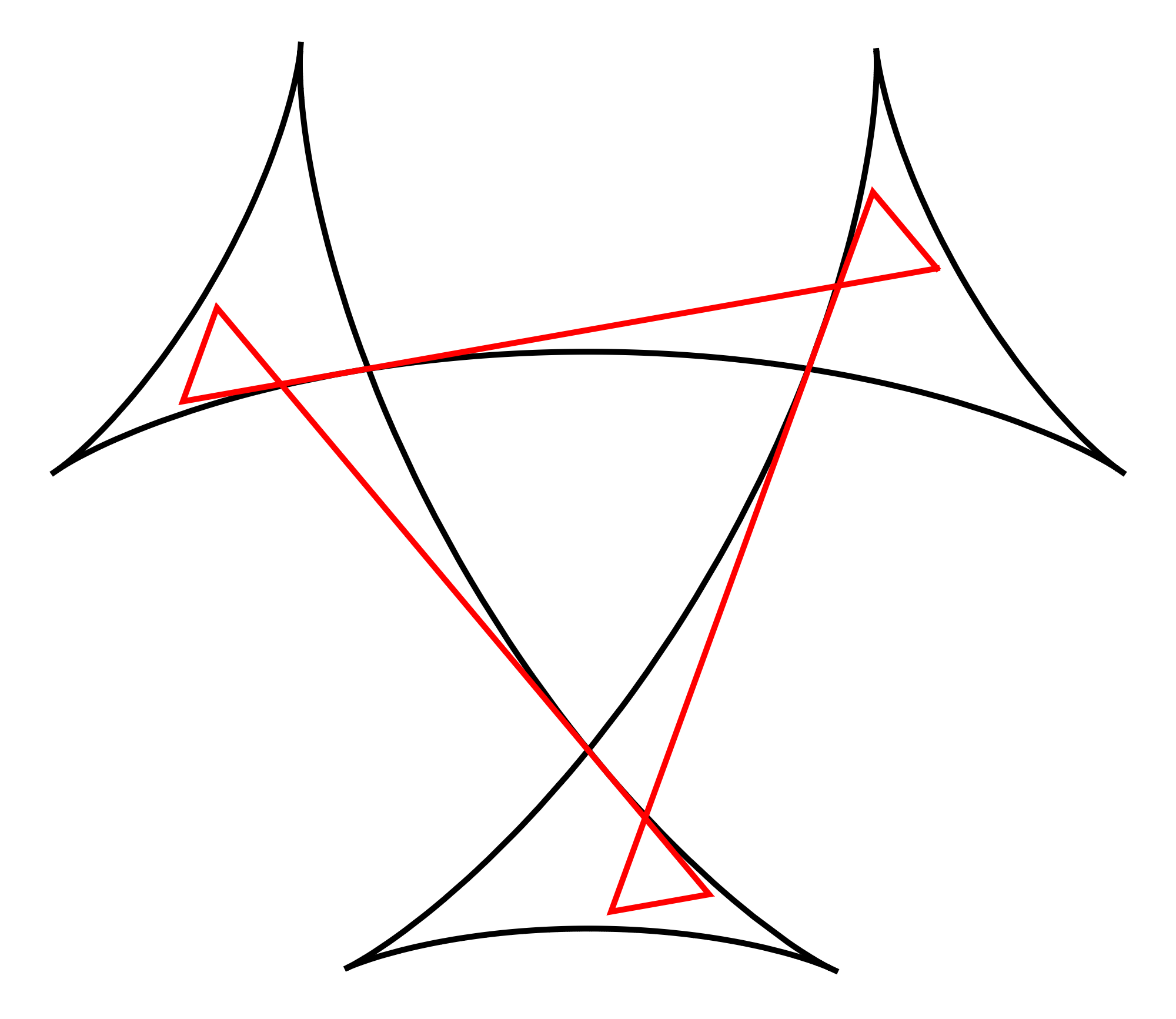}
    \includegraphics[width=0.3\linewidth]{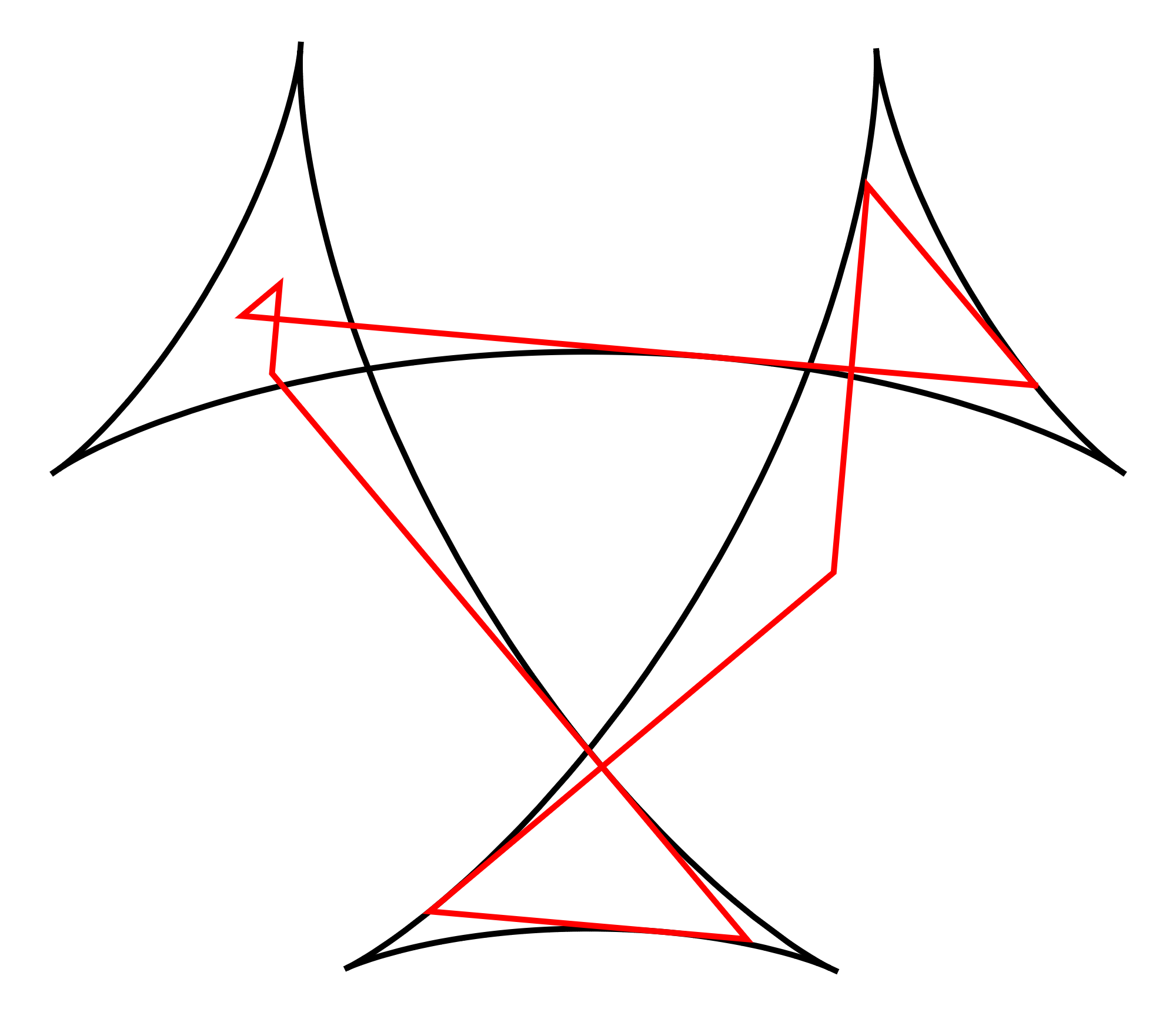}\\
    \includegraphics[width=0.3\linewidth]{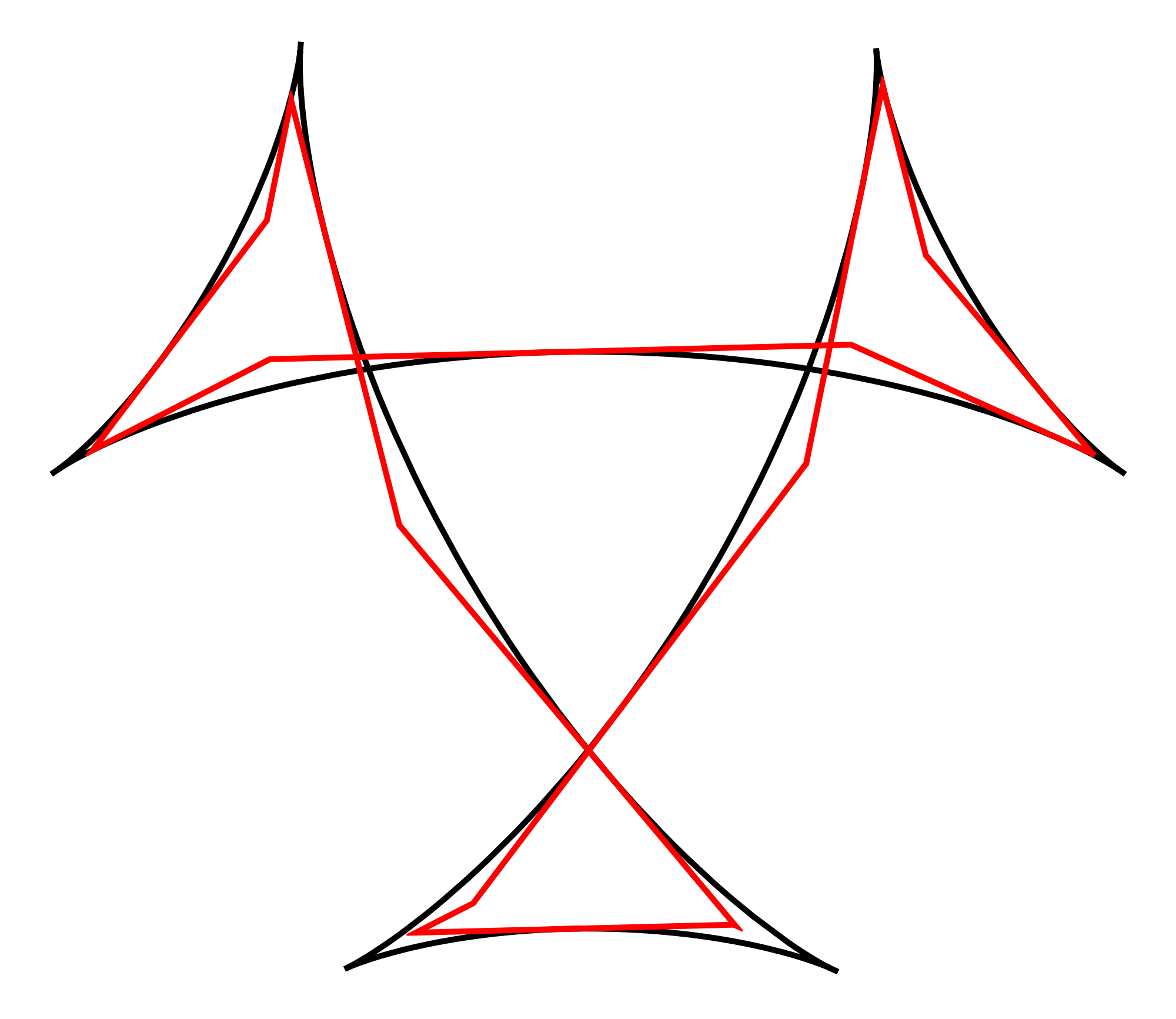}
    \includegraphics[width=0.3\linewidth]{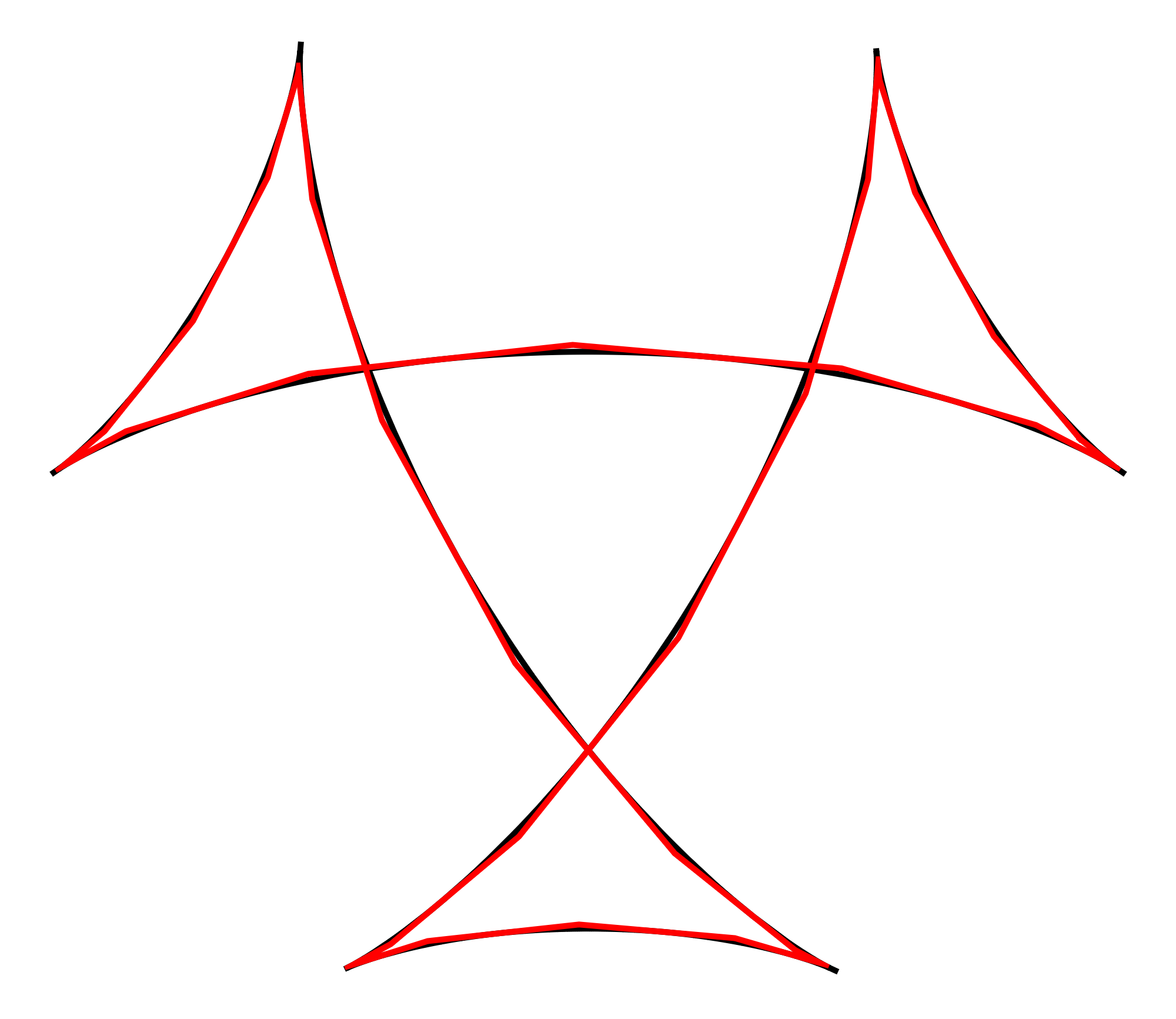}
    \caption{A curve $M$ with six cusp singularities together with $\ppos_\nu(M)$ for $\nu$ with $3$, $4$, $7$, and $16$ directions, respectively}
    \label{FigCPPOSNotReaching}
\end{figure}

\begin{remark}\label{remConstruction}
If $M$ is convex, then $\ppos_\nu(M)$ is a $\cppos$. Conversely, every $\cppos$ can be obtained in this way, for a suitable curve $M$ and a suitable allowable set of directions $\nu$.
\end{remark}

\begin{figure}[h]
    \centering
    \includegraphics[width=0.369\linewidth]{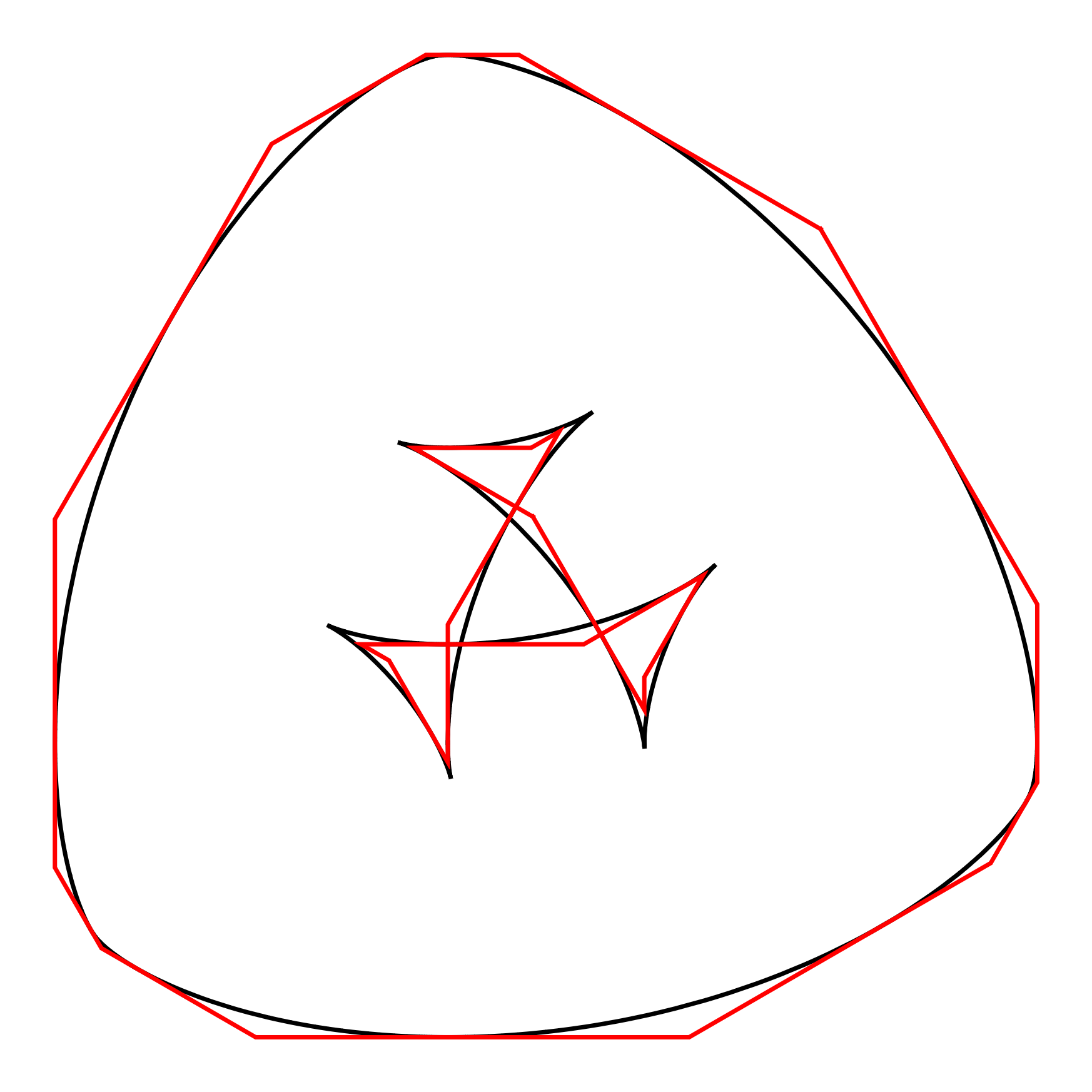}
    \includegraphics[width=0.369\linewidth]{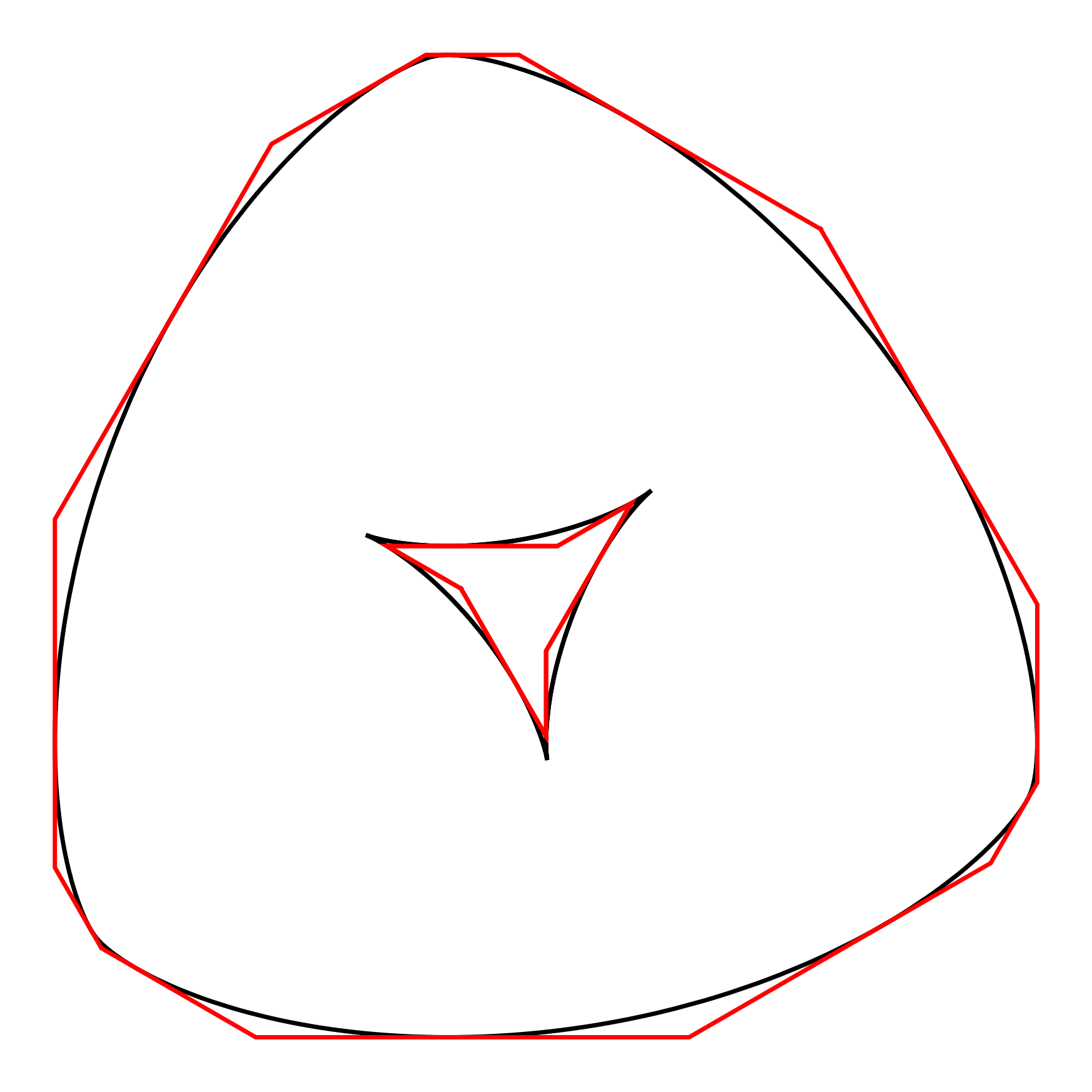}
    \caption{An oval $\mathcal{O}$, $\Eq_{\lambda}(\mathcal{O})$, $\ppos_\nu(\mathcal{O})$, and  $\ppos_\nu\left(\Eq_{\lambda}(\mathcal{O})\right)=\Eq_{\lambda}\left(\ppos_\nu(\mathcal{O})\right)$, where $\lambda=0.4$ and $\lambda=0.5$, respectively}
    \label{FigThmChanging}
\end{figure}

We shall call a closed planar curve an \textit{oval} if it is smooth, strictly convex, and has no self-intersections.

The following theorem is a key observation: constructing a polygonal chain from an affine equidistant of an oval is compatible with taking the discrete equidistant of the polygonal chain constructed from the original oval -- see Figure~\ref{FigThmChanging}.

\begin{thm}
Let $M$ be an oval, let $\nu$ be an allowable set of directions, and let~$\lambda\in\mathbb{R}$, such that the only singularities of $\Eq_\lambda(M)$ are ordinary cusps. Then
\[
\ppos_{\nu}\left(\Eq_{\lambda}(M)\right)
=
\Eq_{\lambda}\left(\ppos_{\nu}(M)\right).
\]
\end{thm}

\begin{proof}
Let $a,b\in M$ be a parallel pair such that the tangent lines to $M$ at $a$ and $b$ have a direction belonging to $\nu$. By Theorem~2.6 in \cite{ZD_Wigner}, the line tangent to $\Eq_{\lambda}(M)$ at the point
$\lambda a+(1-\lambda)b$
is parallel to the tangent lines to $M$ at $a$ and $b$.

This tangent line contains one side of the polygonal chain
\[
\ppos_{\nu}\left(\Eq_{\lambda}(M)\right).
\]
On the other hand, the tangent lines to $M$ at $a$ and $b$ contain two corresponding sides of $\ppos_{\nu}(M)$. Hence, by the definition of the discrete equidistant, the corresponding side of
$\Eq_{\lambda}\left(\ppos_{\nu}(M)\right)$
is parallel to them as well.

Moreover, by the intercept theorem, this side also passes through the point
$(1-\lambda) a+\lambda b$. 
Thus, the two corresponding side-lines of
\[
\ppos_{\nu}\left(\Eq_{\lambda}(M)\right)
\quad\text{and}\quad
\Eq_{\lambda}\left(\ppos_{\nu}(M)\right)
\]
are parallel and pass through the same point. Therefore, they coincide.

Since this argument applies to every direction in $\nu$, the corresponding side-lines of both polygonal chains coincide. Consequently, their consecutive intersection points coincide as well. Hence, 
\[
\ppos_{\nu}\left(\Eq_{\lambda}(M)\right)
=
\Eq_{\lambda}\left(\ppos_{\nu}(M)\right),
\]
which ends the proof.
\end{proof}

\begin{definition}
For a finite set of directions
\[
\nu=\{\alpha_1,\ldots,\alpha_m\}\subset[0,\pi),
\qquad
0\leqslant\alpha_1<\cdots<\alpha_m<\pi,
\]
we define its \textit{angular mesh} by
\[
\operatorname{mesh}(\nu)
:=
\max\left\{
\alpha_2-\alpha_1,\ldots,\alpha_m-\alpha_{m-1},
\alpha_1+\pi-\alpha_m
\right\}.
\]
\end{definition}

\begin{thm}\label{thm:convergence}
Let $M$ be a generic oval, and let $\lambda\in\mathbb{R}$ be generic. If $(\nu_k)_{k\in\mathbb{N}}$ is a~sequence of allowable sets of directions for the affine equidistant $\Eq_{\lambda}(M)$ such that
\[
\operatorname{mesh}(\nu_k)\to 0
\quad\text{as}\quad k\to\infty,
\]
then
\[
\lim_{k\to\infty}
\ppos_{\nu_k}\left(\Eq_{\lambda}(M)\right)
=
\Eq_{\lambda}(M)
\]
in the Hausdorff metric.
\end{thm}

\begin{proof}
Since $M$ is generic, it is known that, for a generic value of $\lambda$, the affine equidistant $\Eq_{\lambda}(M)$ is a piecewise smooth curve whose only singularities are ordinary cusps -- for details and for the precise genericity assumptions, see \cite{ZD_Wigner} and the literature therein.

Remove from $\Eq_{\lambda}(M)$ all its cusp points. The remaining curve is the union of finitely many open smooth arcs. These arcs can be further divided into finitely many overlapping arcs which, in suitable local coordinate systems, are graphs of $C^1$ functions whose first derivatives are uniformly bounded by some constant $L>0$. On each compact regular subarc, the tangent direction is a local coordinate whose inverse is uniformly continuous. Hence, as $\operatorname{mesh}(\nu)\to0$, the distance between consecutive points of tangency tends uniformly to zero.

Let us focus on one such arc. It is represented as the graph of a function
\[
g_i\colon (a_i,b_i)\to\mathbb{R}.
\]
For a sufficiently small mesh, the corresponding part of the polygonal chain
\[
\ppos_{\nu_k}\left(\Eq_{\lambda}(M)\right)
\]
is represented by a piecewise linear function
\[
h_i\colon (a_i,b_i)\to\mathbb{R}
\]
whose segments are tangent to the graph of $g_i$. The first and the last segment may fail to be tangent at points belonging to $(a_i,b_i)$; however, by slightly shrinking the interval, we may assume that all these segments are tangent at points of the considered arc. This can be done in such a way that, for every fixed $\varepsilon>0$, all points of $\Eq_{\lambda}(M)$ and of the polygonal chain which are not represented by some $g_i$ or $h_i$ lie in an $\varepsilon$-neighbourhood of the cusp points. This follows from the local normal form of an ordinary cusp and the continuous dependence of the intersections of nearby tangent lines on their points of tangency.

Fix a set of directions $\nu$ satisfying the above assumptions and put
\[
\mathcal{P}:=\ppos_{\nu}\left(\Eq_{\lambda}(M)\right).
\]
Recall that 
\[
d_H\left(\mathcal{P},\Eq_\lambda(M)\right)
=
\max\left\{
\sup_{y\in \Eq_\lambda(M)}\inf_{z\in\mathcal{P}}\|y-z\|,\ 
\sup_{z\in\mathcal{P}}\inf_{y\in\Eq_\lambda(M)}\|y-z\|
\right\},
\]
where $d_H$ is the Hausdorff metric.

We split $\Eq_\lambda(M)$ and $\mathcal{P}$ into two parts:
\[
M_1
=
\{y\in \Eq_\lambda(M): y=(x,g_i(x)) \text{ for some } i \text{ and some } x\},
\ \ 
M_2=\Eq_\lambda(M)\setminus M_1,
\]
and
\[
\mathcal{P}_1
=
\{z\in \mathcal{P}: z=(x,h_i(x)) \text{ for some } i \text{ and some } x\},
\quad
\mathcal{P}_2=\mathcal{P}\setminus \mathcal{P}_1.
\]

Let $y\in M_1$, say
\[
y=(x,g_i(x)).
\]
Assume that $h_i(x)$ belongs to a segment tangent to the graph of $g_i$ at $x_0$. Then
\[
h_i(x)=g_i(x_0)+g_i'(x_0)(x-x_0).
\]
Hence, by the mean value theorem,
\[
\begin{aligned}
\inf_{z\in\mathcal{P}}\|y-z\|
&\leqslant |g_i(x)-h_i(x)| \\
&\leqslant |g_i(x)-g_i(x_0)|
   + |g_i'(x_0)|\,|x-x_0| \\
&= \bigl(|g_i'(c)|+|g_i'(x_0)|\bigr)|x-x_0| \\
&\leqslant 2L|x-x_0|
\end{aligned}
\]
for some $c$ between $x$ and $x_0$. Since the mesh of $\nu$ tends to $0$, the maximal distance between $x$ and the corresponding point of tangency $x_0$ tends to $0$. Therefore, 
\[
\sup_{y\in M_1}\inf_{z\in\mathcal{P}}\|y-z\|
\to 0.
\]

Now let $y\in M_2$. By construction, $y$ lies in an $\varepsilon$-neighbourhood of a cusp point. Hence, there exists $y_1\in M_1$ such that
\[
\|y-y_1\|\leqslant 2\varepsilon.
\]
Using the previous estimate, we get
\[
\inf_{z\in\mathcal{P}}\|y-z\|
\leqslant
\|y-y_1\|+\inf_{z\in\mathcal{P}}\|y_1-z\|
\leqslant
2\varepsilon+\inf_{z\in\mathcal{P}}\|y_1-z\|.
\]
Thus, 
\[
\sup_{y\in \Eq_\lambda(M)}
\inf_{z\in\mathcal{P}}\|y-z\|
\]
can be made arbitrarily small by first choosing $\varepsilon>0$ sufficiently small and then taking the mesh of $\nu$ sufficiently small.

The reverse estimate is analogous. Indeed, if $z\in\mathcal{P}_1$, say $z=(x,h_i(x))$, then choosing $y=(x,g_i(x))\in \Eq_\lambda(M)$ gives
\[
\inf_{y\in\Eq_\lambda(M)}\|z-y\|
\leqslant
|h_i(x)-g_i(x)|,
\]
which is estimated exactly as above. If $z\in\mathcal{P}_2$, then $z$ lies in an $\varepsilon$-neighbourhood of a cusp point, and the same argument applies.

Consequently,
\[
\sup_{z\in\mathcal{P}}
\inf_{y\in \Eq_\lambda(M)}\|z-y\|
\]
also tends to $0$ as $\operatorname{mesh}(\nu)\to 0$. Therefore, 
\[
d_H\bigl(\ppos_{\nu}(\Eq_\lambda(M)),\Eq_\lambda(M)\bigr)
\to 0,
\]
which proves the theorem.
\end{proof}

The above convergence theorem will allow us to transfer some properties from the polygonal setting to the smooth one. More precisely, certain statements which can be proved in an elementary way for $\cppos$es may be applied to the polygonal chains $\ppos_{\nu_k}(\Eq_{\lambda}(M))$ and then passed to the limit as $\operatorname{mesh}(\nu_k)\to 0$. In this way, we obtain limiting versions of these properties for affine equidistants of smooth curves.

\begin{definition}\label{Def:cusp}
Let $\mathcal{P}$ be a $\cppos$, and let $P_i(\lambda)$ be a vertex of $\Eq_{\lambda}(\mathcal{P})$. We say that $P_i(\lambda)$ is a \textit{cusp} formed by the adjacent edges $e_{i-1}(\lambda)$ and $e_i(\lambda)$ if
\[
\left\langle e_{i-1}(\lambda),e_i(\lambda)\right\rangle < 0.
\]

\end{definition}

\begin{prop}\label{cusp_equivalence}
Let $\mathcal{P}$ be a $\cppos$. For a vertex $P_i(\lambda)$ of $\Eq_{\lambda}(\mathcal{P})$, the following conditions are equivalent:
\begin{enumerate}[(a)]
\item The edges $e_{i-1}(\lambda)$ and $e_i(\lambda)$ form a cusp.
\item $\left[e_{i-1}(\lambda),e_i(\lambda)\right] < 0$.
\item 
$\left(\lambda_{i-1}-\lambda\right)
\left(\lambda_i-\lambda\right)<0.$
\item %The angle between the two edges that meet at $P_i(\lambda)$ is acute. 
The unoriented angle between the two adjacent edges at their common vertex $P_i(\lambda)$, that is, the angle between $-e_{i-1}(\lambda)$ and $e_i(\lambda)$, is acute.
\end{enumerate}
\end{prop}

\begin{proof}
By \eqref{equation2 krawedzie remark}, we have
\[
e_j(\lambda)
=
\frac{\lambda_j-\lambda}{\lambda_j}e_j.
\]
Therefore, 
\[
\left[e_{i-1}(\lambda),e_i(\lambda)\right]
=
\frac{(\lambda_{i-1}-\lambda)(\lambda_i-\lambda)}
{\lambda_{i-1}\lambda_i}
\left[e_{i-1},e_i\right],
\]
and similarly
\[
\left\langle e_{i-1}(\lambda),e_i(\lambda)\right\rangle
=
\frac{(\lambda_{i-1}-\lambda)(\lambda_i-\lambda)}
{\lambda_{i-1}\lambda_i}
\left\langle e_{i-1},e_i\right\rangle.
\]
Since $\mathcal{P}$ is a $\cppos$, we have
\[
\lambda_{i-1}>0,
\qquad
\lambda_i>0,
\]
and, by the assumed geometry of $\cppos$es,
\[
\left[e_{i-1},e_i\right]>0,
\qquad
\left\langle e_{i-1},e_i\right\rangle>0.
\]
Hence, the signs of both
\[
\left[e_{i-1}(\lambda),e_i(\lambda)\right]
\quad\text{and}\quad
\left\langle e_{i-1}(\lambda),e_i(\lambda)\right\rangle
\]
are equal to the sign of
\[
(\lambda_{i-1}-\lambda)(\lambda_i-\lambda).
\]
This proves the equivalence of (a), (b), and (c).

Finally, condition (a) is equivalent to (d), because
\[
\left\langle e_{i-1}(\lambda),e_i(\lambda)\right\rangle<0
\]
means that the angle between the oriented vectors $e_{i-1}(\lambda)$ and $e_i(\lambda)$ is obtuse, and therefore the angle between the two edges that meet at $P_i(\lambda)$, namely the~unoriented angle between $-e_{i-1}(\lambda)$ and $e_i(\lambda)$, is acute.
\end{proof}

\begin{remark}
In \cite{CraizerPPOS} the definition of a cusp of a $\lambda$-equidistant is (c) from Proposition 2.15  for $\lambda=0.5$, and (b) otherwise.
\end{remark}

\begin{thm}
Let $\mathcal{P}$ be a $\cppos$. For a generic value $\lambda\neq 0.5$, the number of cusps of $\Eq_{\lambda}(\mathcal{P})$ is even. If $\lambda= 0.5$ is generic, then the number of cusps of $\Eq_{0.5}(\mathcal{P})$ is odd.
\end{thm}

\begin{proof}
For a $\cppos$ $\mathcal{P}$, one can define a unique outward unit normal vector field along the edges. This field is constant on each edge and is undefined at the vertices. Since two edges sharing a vertex form an obtuse angle, the angle between their outward normal vectors is acute.

\begin{figure}[h]
    \centering
    \includegraphics[width=0.666
    \linewidth]{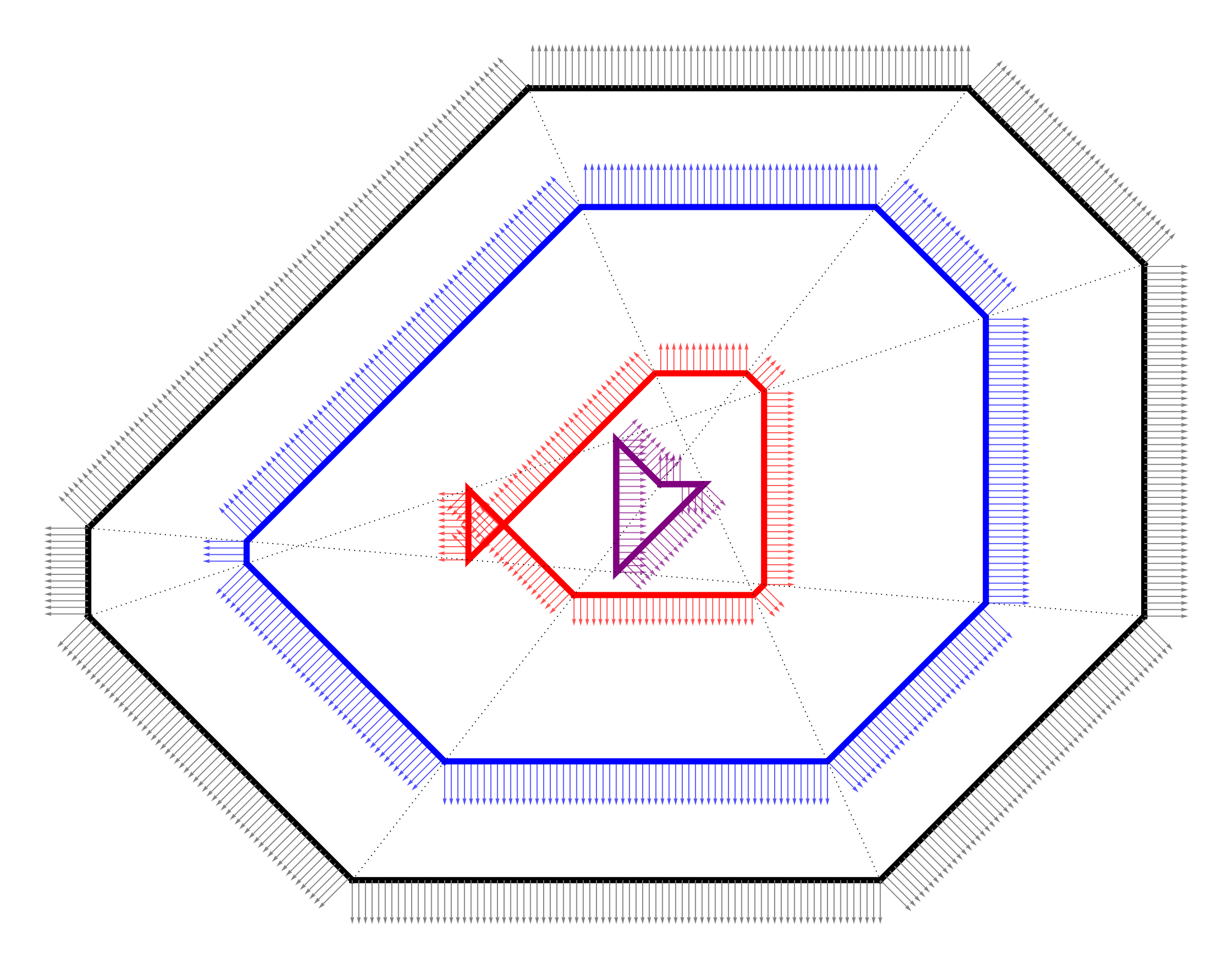}
    \caption{Normal vector fields along equidistants}
    \label{Fig:normalVectorFields}
\end{figure}

We transfer this normal field to the equidistant $\Eq_{\lambda}(\mathcal{P})$ by means of the natural correspondence between the edges of $\mathcal{P}$ and the edges of $\Eq_{\lambda}(\mathcal{P})$, given by the~homothety of each edge. Then, whenever a cusp is formed, the normal vector points inside the acute angle on one adjacent edge and outside it on the other (see Figure \ref{Fig:normalVectorFields}). The key observation is that, if a cusp vertex is surrounded by non-cusp vertices, then the adjacent edges bend away from it, and not towards it, just as in the cusp of a smooth curve.

Thus, when passing through consecutive cusps, the position of the normal vector alternates: if at one cusp it changes from pointing inside the acute angle to pointing outside it, then at the next cusp it changes in the opposite way, and vice versa. Moreover, when passing through a cusp, the orientation of the ordered pair consisting of the edge vector of the Wigner caustic, or more generally of an~affine equidistant, taken consistently with the orientation of the corresponding set, and the normal vector inherited from the original $\cppos$, is reversed -- this local behaviour is illustrated in Figure~\ref{Fig:normalVectors}.

\begin{figure}[h]
    \centering
    \includegraphics[width=0.6
    \linewidth]{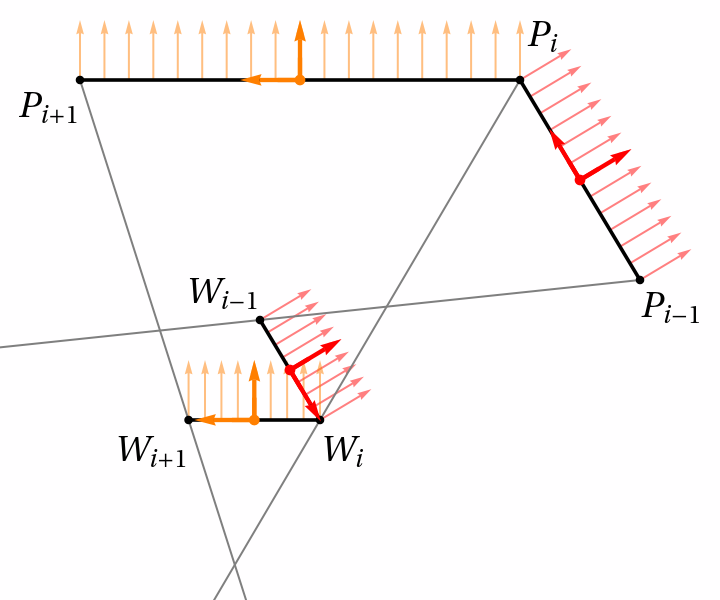}
    \caption{Normal vector fields along two sides of a $\cppos$ and its Wigner caustic}
    \label{Fig:normalVectors}
\end{figure}

Therefore, for $\lambda\neq 0.5$, we can divide the edges of $\Eq_{\lambda}(\mathcal{P})$ into two classes, say $A$ and $B$, as follows. We declare $e_1(\lambda)$ to be of type $A$, and then we say that $e_{i+1}(\lambda)$ is of the same type as $e_i(\lambda)$ if and only if the two edges do not form a cusp. Equivalently, the type changes exactly when we pass through a cusp.

Since $\Eq_{\lambda}(\mathcal{P})$ has $2n$ edges and after a full cycle we return to the initial edge, the type must also return to the initial one. Hence, the number of changes of type, that is, the number of cusps, must be even.

For $\lambda=0.5$, the same argument applies, but now $\Eq_{0.5}(\mathcal{P})$ has only $n$ distinct sides. Hence, after one full cycle on the Wigner caustic, we have made only half a cycle on the original polygon. Therefore, the corresponding normal direction is reversed (see Figure \ref{Fig:normalVectorFields}), and the type of the edge must change after one full cycle. Consequently, the number of type changes, and hence, the number of cusps, must be odd.
\end{proof}

\begin{remark}
The same result can be obtained from the alternative characterization given in condition~(c) of Proposition~\ref{cusp_equivalence}. Let $\lambda\neq 0.5$. If we consider the sequence $\left(\lambda_i\right)_{i\in\mathbb{Z}}$, then a cusp occurs every time this sequence crosses the value $\lambda$. Since $\lambda$ is generic, the sequence never takes the value $\lambda$. But the sequence is cyclic, and therefore, in order to return to its original value after $2n$ terms, it must cross the level $\lambda$ an even number of times.

For $\lambda=0.5$, the relation
\[
\lambda_{i+n}=1-\lambda_i
\]
implies that, after $n$ terms, the sign of $\lambda_i-0.5$ is reversed. Hence, the number of crossings of the level $0.5$ along one cycle of $\Eq_{0.5}(\mathcal{P})$ is odd. This method was used in \cite{CraizerPPOS} to prove the result for $\lambda=0.5$.
\end{remark}

\begin{definition}
Let $M$ be a smooth regular closed planar curve. For every parallel pair $(a,b)$, let $\ell_{a,b}$ denote the line passing through $a$ and $b$.
The \emph{centre symmetry set} of $M$, denoted by $\operatorname{CSS}(M)$, is the envelope of the family of lines
$$
\bigl\{\ell_{a,b}:\ (a,b)\text{ is a parallel pair of }M\bigr\}.
$$
Thus, $\operatorname{CSS}(M)$ is the envelope of the chords joining pairs of points of $M$ with parallel tangent lines.
\end{definition}

\begin{figure}[h!]
    \centering
    \includegraphics[width=0.555
    \linewidth]{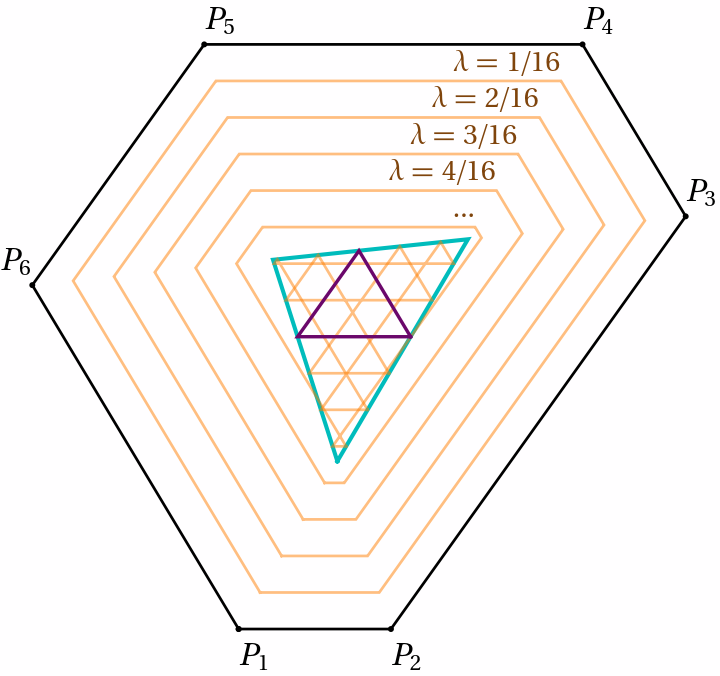}
    \caption{Family of $\lambda$-equidistants together with the centre symmetry set of a $\cppos$}
    \label{fig:LambasCss}
\end{figure}

The study of the centre symmetry set originates in the work of S. Janeczko~\cite{Janeczko}, who described it as the bifurcation set of a suitable family of ratio functions. This notion was subsequently reformulated by P. J. Giblin and P. A. Holtom in the geometric form adopted above \cite{GiblinHoltom}. Symmetry and its various generalizations play a fundamental role in the geometry of submanifolds of Euclidean spaces and in numerous related applications.
The centre symmetry set of a curve may also be characterized as the union of the singular points of its family of affine $\lambda$-equidistants. This interpretation has been extensively studied both in the smooth setting and for non-smooth objects, including polygonal curves. An analogous result concerning singularities was established by the authors of \cite{CraizerPPOS}. More precisely, they proved that the relative interior of every edge of the centre symmetry set of a $\cppos$ consists entirely of singular points of the family of affine $\lambda$-equidistants (see Proposition 3.1 in \cite{CraizerPPOS}). Thus, apart from the endpoints of its edges, the polygonal centre symmetry set is formed entirely by singular points arising within this family (see Figure \ref{fig:LambasCss}). The geometry of the centre symmetry set and its singular structure has attracted considerable attention, not only in the plane but also in higher-dimensional Euclidean spaces and in Minkowski geometry \cite{Art, DomitrzRios, GB, GiblinZ2, Dominika, ZwierzMix}.

\begin{remark}
The exact compatibility established in Theorem~2.12 is a feature peculiar to $\lambda$-equidistants. In particular, no analogous identity holds for the centre symmetry set. In general,
\begin{equation*}
\ppos_{\nu'}(\Css(M)) \neq \Css(\ppos_\nu(M)),
\end{equation*}
where $\nu$ and $\nu'$ are suitable allowable sets of directions for $M$ and $\Css(M)$, respectively. Figure~\ref{fig:csscppsnot} illustrates the typical situation: the tangent lines to the centre symmetry set of an oval $\mathcal{O}$ do not coincide with the lines containing the edges of $\Css(\ppos_\nu(\mathcal{O}))$.

\begin{figure}[h!]
    \centering
    \includegraphics[width=0.405
    \linewidth]{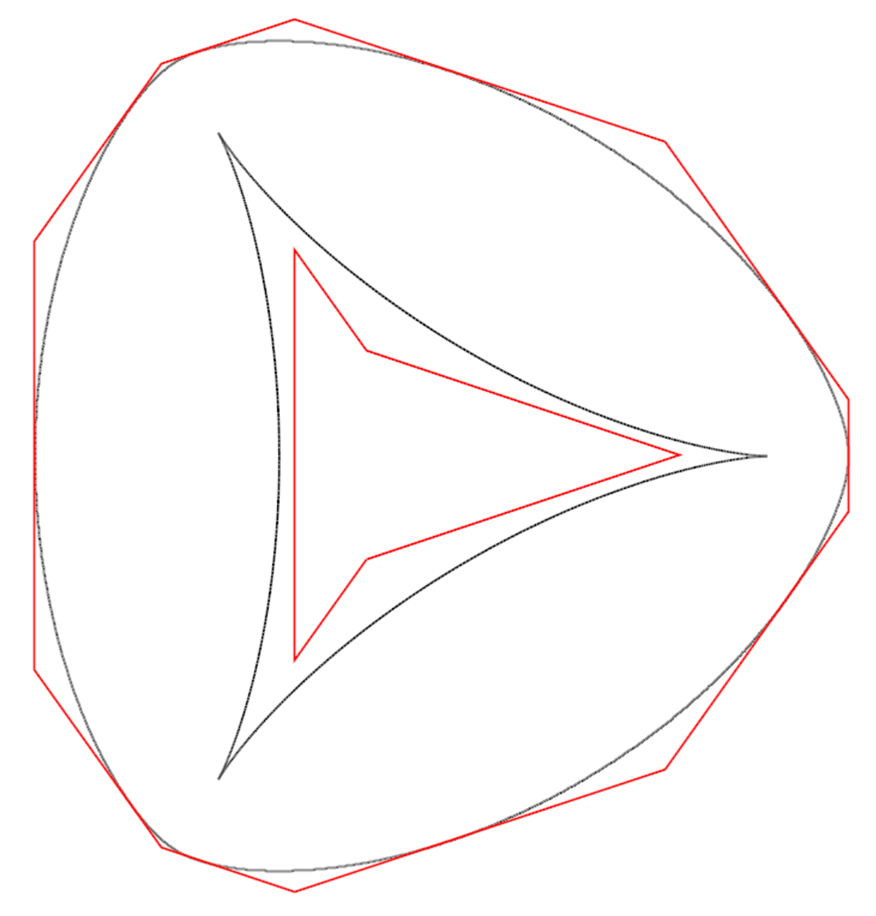}
    \includegraphics[width=0.405
    \linewidth]{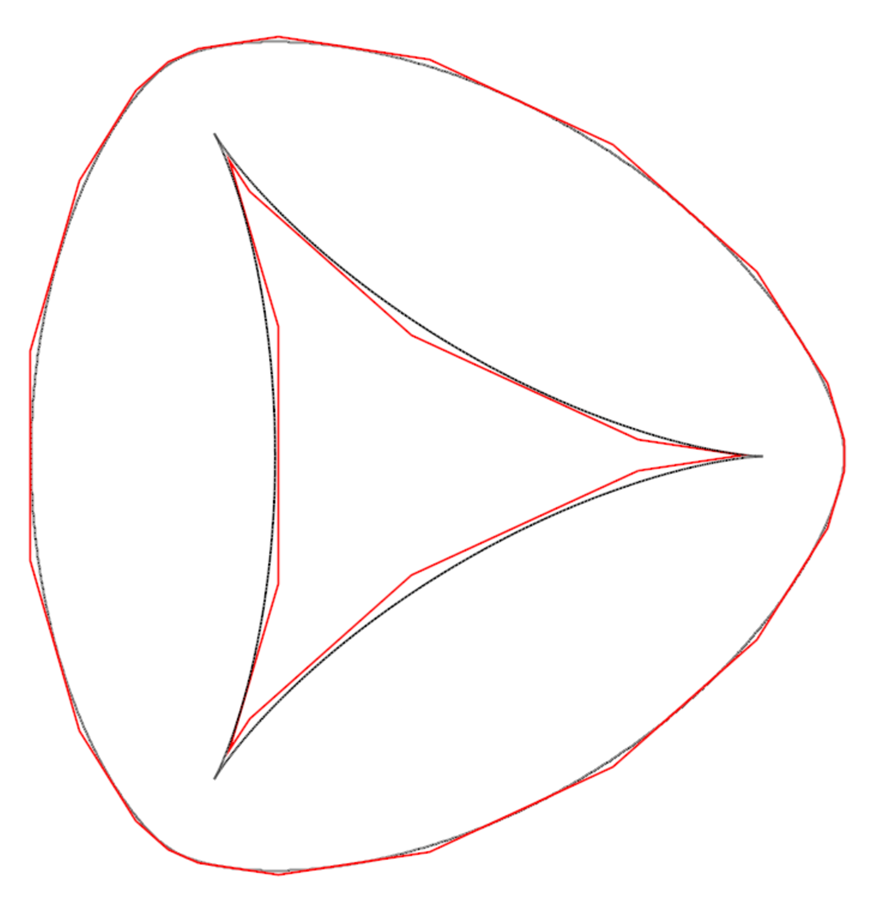}
    \caption{An oval $\mathcal{O}$, its centre symmetry set $\Css(\mathcal{O})$, the polygons $\ppos_\nu(\mathcal{O})$, and the corresponding centre symmetry sets $\Css(\ppos_\nu(\mathcal{O}))$. In the left-hand figure, the set of directions $\nu$ consists of $5$ directions, whereas in the right-hand figure it consists of $11$ directions}
    \label{fig:csscppsnot}
\end{figure}

Nevertheless, the discrete centre symmetry sets approximate the smooth one as the mesh of $\nu$ tends to zero. Thus, in this case, exact commutation is replaced by convergence in the Hausdorff metric.
\end{remark}

\begin{prop}\label{prop:css-hausdorff-convergence}
Let $M$ be a generic oval and let $(\nu_k)_{k\in\mathbb{N}}$ be a sequence
of allowable sets of directions such that
\begin{equation*}
\operatorname{mesh}(\nu_k) \to 0
\quad \text{as} \quad k \to \infty .
\end{equation*}
Then
\begin{equation*}
\Css(\ppos_{\nu_k}(M)) \to \Css(M)
\quad \text{as} \quad k \to \infty
\end{equation*}
in the Hausdorff metric.
\end{prop}

\begin{proof}
Let $\ell_\theta$ denote the affine chord of $M$ corresponding to the pair
of points whose tangent lines have direction $\theta$. The centre symmetry set
$\Css(M)$ is the envelope of the one-parameter family $(\ell_\theta)$.

On every compact regular arc of $\Css(M)$, the envelope can be obtained as the
limit of intersections of two nearby chords,
\begin{equation*}
\ell_\theta \cap \ell_{\theta+h}
\quad \text{as} \quad h \to 0 .
\end{equation*}
The great diagonals of $\ppos_{\nu_k}(M)$ converge uniformly, as $k\to\infty$,
to the corresponding affine chords of $M$. Hence, the vertices of
$\Css(\ppos_{\nu_k}(M))$, which are intersections of consecutive great
diagonals, converge uniformly on regular arcs to the corresponding points of
$\Css(M)$.

Since $\Css(M)$ is compact and has only finitely many cusp points, we can cover its regular part by finitely many such arcs and take arbitrarily small neighbourhoods of the cusps. By the local normal form of an ordinary cusp and the continuous dependence of the affine chords on their tangent directions, the intersections of consecutive chords corresponding to sufficiently close directions remain in these neighbourhoods. The polygonal edges joining such consecutive intersections remain there as well. Together with the uniform convergence on the regular arcs, this implies that
\begin{align*}
d_H\bigl(\Css(\ppos_{\nu_k}(M)), \Css(M)\bigr) \to 0
\quad \text{as} \quad k \to \infty .
\end{align*}
\end{proof}

%%%%%%%%%%%%%%%%%%%%%%%%%%%%%%%%%%%%%%%%%%%%%%%%%%%%%
%%%%%%%%%%%%%%%%%%%%%%%%%%%%%%%%%%%%%%%%%%%%%%%%%%%%%%
%%%%%%%%%%%%%%%%%%%%%%%%%%%%%%%%%%%%%%%%%%%%%%%%%%%%%%
\section{Reconstruction of $\cppos$}

\noindent In this section we address an inverse problem for $\cppos$es. Instead of studying the affine equidistants associated with a given polygon, we ask to what extent the original polygon can be recovered from such derived objects. This question is natural in the affine setting, since the Wigner caustic and the centre symmetry set encode information about the great diagonals, midpoints of opposite vertices, and intersections of corresponding affine chords.

\begin{remark}
The definitions of the affine $\lambda$-equidistant and the centre symmetry set remain meaningful for closed polygonal chains with $2n$ edges such that each pair of edges whose indices differ by $n$ is parallel. In this more general setting, the convexity assumption is omitted.
\end{remark}

We begin with a simple but useful observation showing that the class of affine equidistants of a fixed $\cppos$ is closed under taking equidistants once again. More precisely, taking a $\delta$-equidistant of $\Eq_\lambda(\mathcal{P})$ gives another affine equidistant of the original polygon $\mathcal{P}$. As a consequence, whenever $\lambda\neq 0.5$, the polygon $\mathcal{P}$ can be recovered from its $\lambda$-equidistant. This provides the first reconstruction result and also explains why the Wigner caustic is preserved along the family of affine equidistants.

\begin{prop}
\label{2.23}
Let $\mathcal{P}$ be a $\cppos$. For every generic value $\lambda\neq 0.5$ and every $\delta\in\mathbb{R}$, the following relation holds:
\begin{align}
\label{eq:eqlofeql}    \Eq_{\delta}(\Eq_{\lambda}(\mathcal{P}))=\Eq
_{\lambda(1-\delta)+\delta(1-\lambda)}(\mathcal{P}).
\end{align}
\end{prop}
\begin{proof}
Using equation \eqref{equation1 krawedzie remark} twice, we can see that the formula for an edge of the left-hand side of \eqref{eq:eqlofeql} is
\begin{align*}
\left(e_i(\lambda)\right)(\delta)
&=(1-\delta)e_i(\lambda)+\delta e_{i+n}(\lambda)\\ 
&=(1-\delta)\left((1-\lambda)e_i+\lambda e_{i+n}\right)+\delta\left((1-\lambda)e_{i+n}+\lambda e_i\right)\\
&=(1-\delta-\lambda+2\delta\lambda)e_i+(\lambda+\delta-2\delta\lambda)e_{i+n},
\end{align*}
which is the same formula with $\lambda(1-\delta)+\delta(1-\lambda)=\lambda+\delta-2\delta\lambda$ instead of $\lambda$. Thus, the corresponding edge vectors of the two polygonal chains coincide, and hence the chains agree up to a translation. Performing the same computation for any one pair of corresponding vertices shows that this translation is zero. Therefore, the two polygonal chains coincide, which completes the proof.
\end{proof}
This proposition can be used to reconstruct the original polygon if given only the $\lambda$-equidistant. It also shows that the discrete Wigner caustic of the equidistant is the same as that of the original polygon.
\begin{cor}
 Taking $\delta=-\lambda/(1-2\lambda)$ and $\delta=0.5$ in Proposition \ref{2.23}, we get the following results:
\begin{align}
\Eq_{-\lambda/(1-2\lambda)}(\Eq_{\lambda}(\mathcal{P}))&=\Eq_0(\mathcal{P})=\mathcal{P},\\ 
\label{eq:wcthesame}\Eq_{0.5}(\Eq_{\lambda}(\mathcal{P}))&=\Eq_{0.5}(\mathcal{P}).
\end{align}
\end{cor}

Observe that all affine equidistants of a $\cppos$ $\mathcal{P}$, with the exception of the Wigner caustic, have the same family of great diagonals. Consequently,
\begin{align}
\label{eq:cssthesame}
\Css(\Eq_{\lambda}(\mathcal{P}))=\Css(\mathcal{P})
\end{align}
for all $\lambda\neq 0.5$.

Therefore, by equations \eqref{eq:wcthesame} and \eqref{eq:cssthesame}, the reconstruction from the Wigner caustic and the centre symmetry set does not distinguish between these affine equidistants. More precisely, it reconstructs a polygonal chain that belongs to the family of affine equidistants of a certain $\cppos$.

Now, we consider a potential Wigner caustic and a symmetry set for which we know that a $\cppos$ exists. We present the reconstruction procedure for $\cppos$ step by step.

\begin{enumerate}
\item We start with the Wigner caustic, the centre symmetry set, and a point that belongs to the reconstructed polygon. We draw extensions of the edges of the centre symmetry set. The vertices of the reconstructed polygon will lie on these extended lines.
\item Since we already know on which lines the vertices of the polygon lie, we can begin by connecting the two vertices of the first edge of the polygon. We connect the points so that the resulting segment is parallel to the corresponding edge of the Wigner caustic. At the same time, we can mark the corresponding edge of the polygon, since the vertices of the Wigner caustic are midpoints between the vertices of the polygon lying on the same great diagonal.
\item We proceed with the next step of the procedure. We choose the adjacent edge of the Wigner caustic and construct segments parallel to it, ending on the previously determined lines. We repeat this process until all edges of the polygon have been constructed.
\end{enumerate}

\pagebreak

\begin{figure}[h!]
\centering
\begin{subfigure}[b]{0.4\textwidth}
\centering
\includegraphics[width=0.85\linewidth]{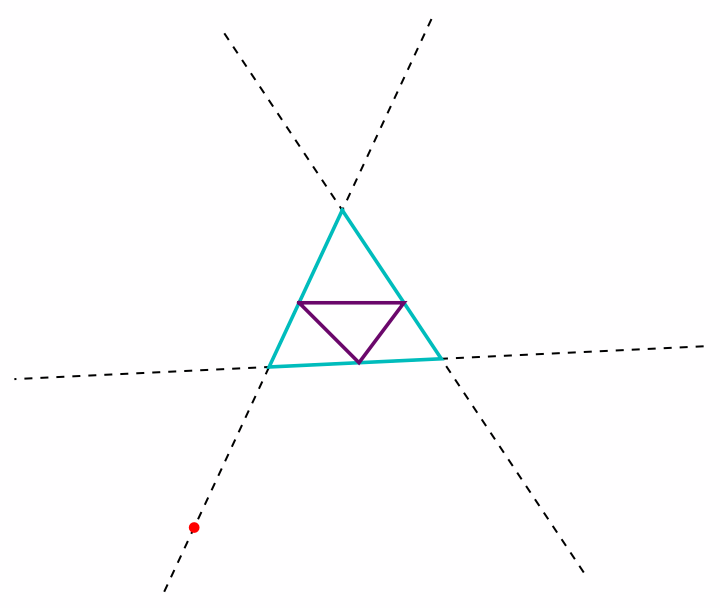}
\subcaption{We draw the main diagonals}
\end{subfigure}
\hfill
\begin{subfigure}[b]{0.4\textwidth}
\centering
\includegraphics[width=0.85\linewidth]{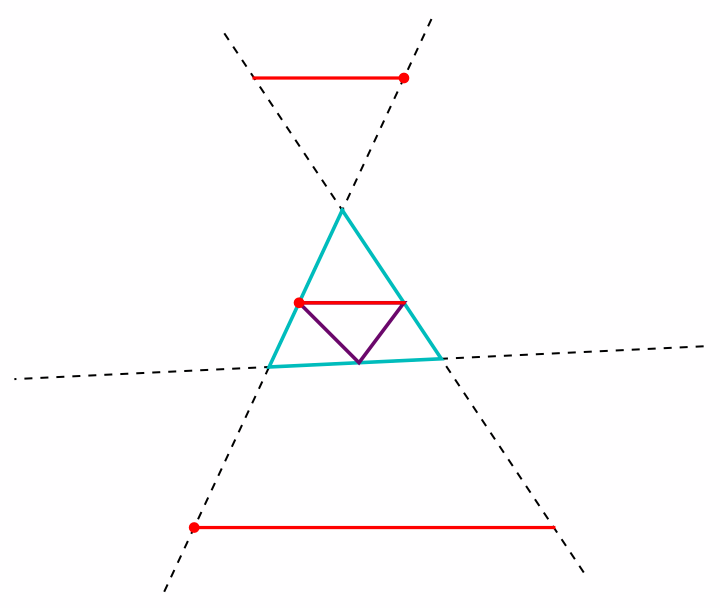}
\subcaption{We mark the first pair of parallel sides}
\end{subfigure}
\
\begin{subfigure}[b]{0.4\textwidth}
\centering
\includegraphics[width=0.85\linewidth]{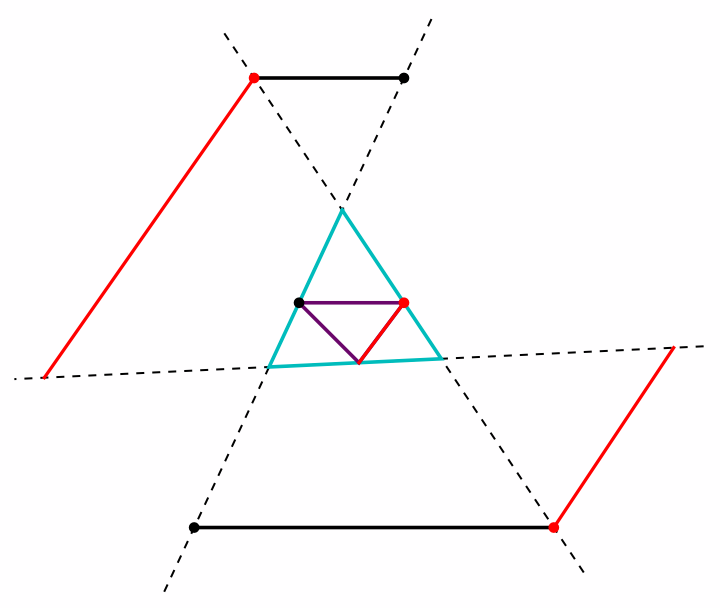}
\subcaption{We mark the second pair of parallel sides}
\end{subfigure}
\hfill
\begin{subfigure}[b]{0.4\textwidth}
\centering
\includegraphics[width=0.85\linewidth]{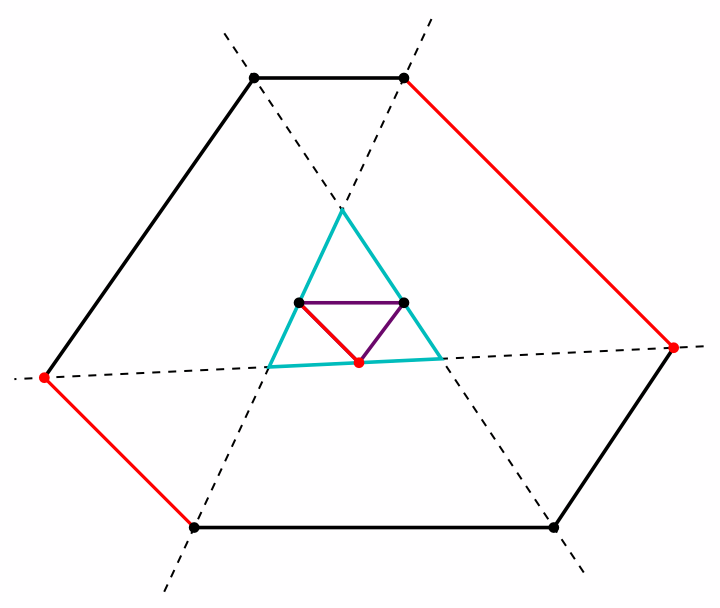}
\subcaption{We mark the third pair of parallel sides}
\end{subfigure}
\caption{Reconstruction of a hexagon from its centre symmetry set, Wigner caustic, and a single point}
\label{hexagonReconstruction}
\end{figure}

\begin{figure}[h!]
    \centering
    \includegraphics[width=0.37
    \linewidth]{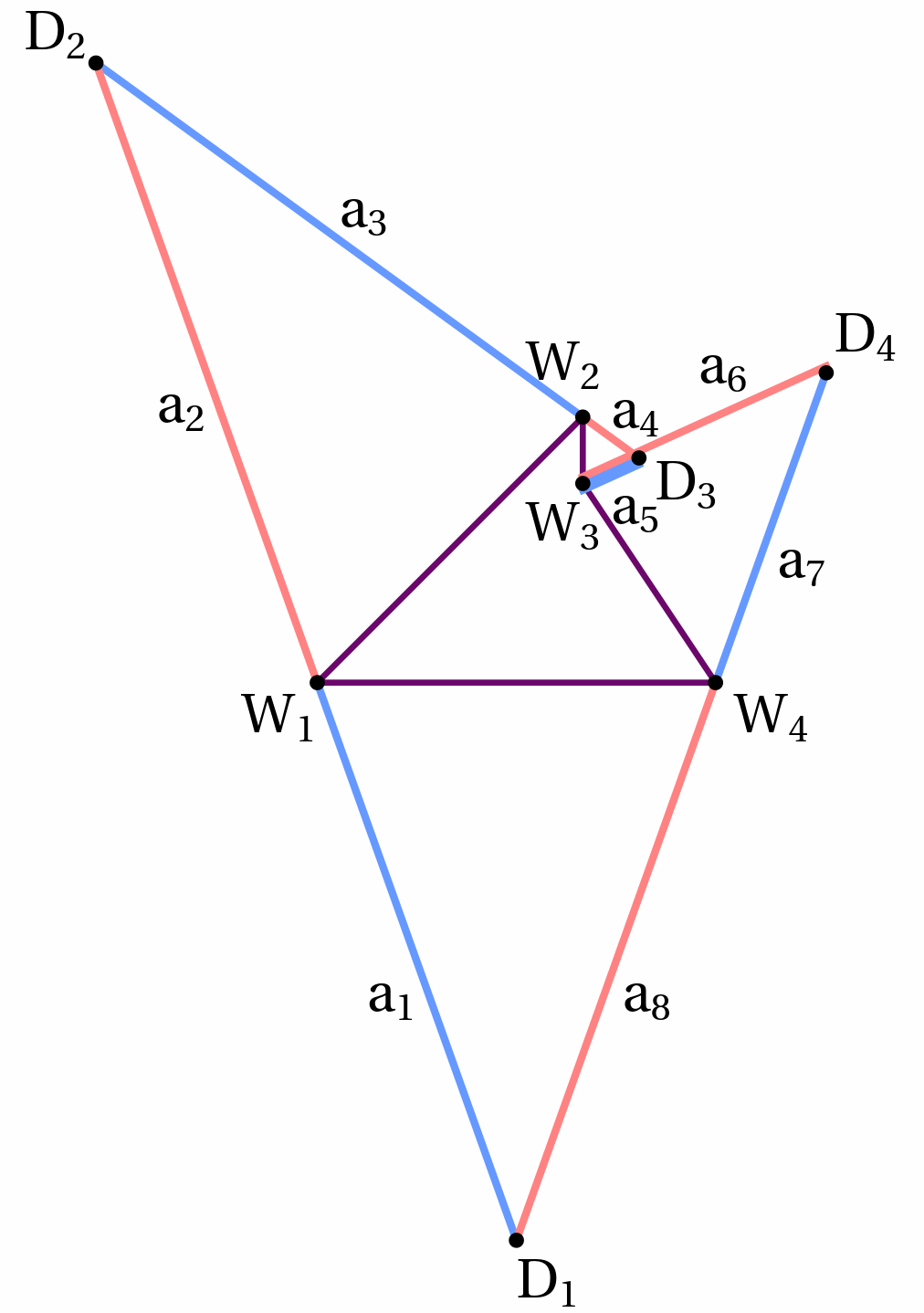}
    \caption{Labeling of the segments determined by the vertices of the potential Wigner caustic and the potential centre symmetry set}
    \label{segments}
\end{figure}

\pagebreak

The above step-by-step reconstruction for a certain hexagon is shown in Figure~\ref{hexagonReconstruction}.

At this stage, we allow the reconstructed polygonal chain to have coincident vertices or degenerate edges, provided that the given reconstruction data prescribe the lines containing its great diagonals.

We begin with the reconstruction of an arbitrary closed polygonal chain, not necessarily a $\cppos$. Let the vertices of a potential Wigner caustic be denoted by $W_1,\ldots, W_n$. Then, for each $i\in\{1,2,\ldots,n\}$, the $i$-th vertex of the Wigner caustic lies on the line determined by the points $D_i$ and $D_{i+1}$, which are vertices of a~potential symmetry set. The potential vertices $W_i$ and $D_i$ are denoted in such a~way that $W_i \in D_iD_{i+1}$. Next, consider the segment $D_iW_i$ and denote it by $a_{2i-1}$. Similarly, denote the segment $W_iD_{i+1}$ by $a_{2i}$. Figure \ref{segments} shows this notation for an example of a potential Wigner caustic and centre symmetry set. Notice that the consecutive segments can overlap like the segments $a_5=D_3W_3$ and $a_6=W_3D_4$ in Figure \ref{segments}.

If the constructed segments $a_i$, for $i\in\{1,\ldots,2n\}$, are allowed to have arbitrary lengths, then reconstruction of a closed polygonal chain is not always possible. This is because the last vertex may fail to coincide with the first one, as illustrated in~Figure~\ref{unclosed}.

\begin{figure}[h!]
    \centering
    \includegraphics[width=0.500
    \linewidth]{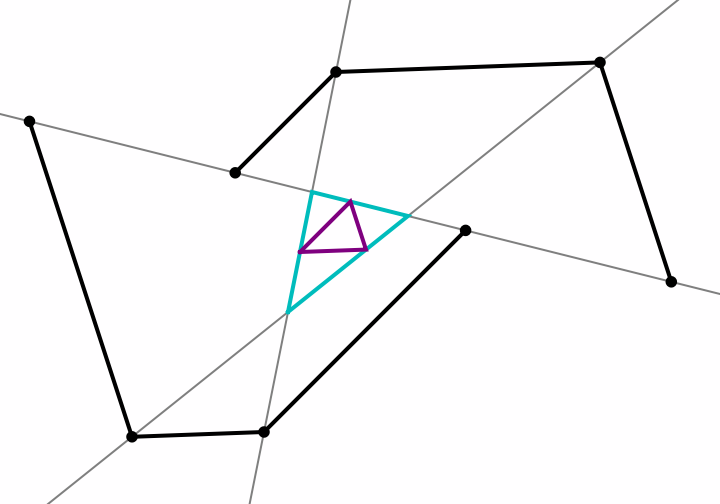}
    \caption{Polygonal chain whose first and last vertices do not coincide}
    \label{unclosed}
\end{figure}

\begin{thm}\label{thmReconstruction}
    Assume that all segment lengths $|a_i|$, for $i\in\{1,2,\ldots,2n\}$, are positive. The reconstruction of an arbitrary closed polygonal chain with $2n$ vertices is possible if and only if
    \begin{align}\label{eq:reconstruction}
    |a_1|\cdot |a_3|\cdot\ldots\cdot |a_{2n-1}|
    =
    |a_2|\cdot |a_4|\cdot\ldots\cdot |a_{2n}|.
    \end{align}
\end{thm}

\begin{proof}
First, let us consider a potential Wigner caustic and a potential centre symmetry set for which we know that a reconstruction is possible. Now we denote the vertices of the closed polygonal chain as $P_1, \dots, P_{2n}$ such that $P_i, P_{n+i}$ are on the line passing through $D_i$,$D_{i+1}$ and the segments $P_iP_{i+1}$ are the sides of the polygon for $i \in \{1,2,\ldots,n\}$, where $D_1$ lies on the segment $W_1P_1$. We will consider triplets of points: $P_{i-1},P_i,P_{i+1}$ and their counterparts $P_{n+i-1},P_{n+i},P_{n+i+1}$. Now we will consider 3 possibilities (the cases in which $D_i$ and $D_{i+1}$ are swapped are analogous):
\definecolor{pointblue}{RGB}{0,55,145}
\definecolor{diagramred}{RGB}{205,25,35}
\definecolor{diagramgreen}{RGB}{0,105,55}
\definecolor{diagramviolet}{RGB}{120,45,160}
\begin{enumerate}[(1)]
\item \textbf{$W_i$ is on the segment $P_iD_i$} (see Figure \ref{fig:ReconstructionProof1})

\begin{figure}[h]
    \centering
    \resizebox{0.79\textwidth}{!}{
    \begin{tikzpicture}[
    x=0.85cm,
    y=0.85cm,
    line cap=round,
    line join=round,
    polygon/.style={
        draw=black,
        line width=1.25pt
    },
    diagonal/.style={
        draw=black,
        line width=0.95pt,
        dash pattern=on 8pt off 6pt
    },
    cssedge/.style={
        draw=black,
        line width=1.35pt
    },
    wigner/.style={
        draw=diagramviolet,
        line width=1.7pt
    },
    bluelabel/.style={
        text=pointblue,
        font=\large
    },
    violetlabel/.style={
        text=diagramviolet,
        font=\large
    }
]
\coordinate (Pimone) at (0,0);
\coordinate (Pi) at (9.4,0);
\coordinate (Piplusone) at (11.7,4.15);

\coordinate (Pniminusone) at (8.55,9.70);
\coordinate (Pni) at (2.75,9.70);
\coordinate (Pniplusone) at (1.715,7.8325);

\coordinate (Di) at (5.2875,5.9987);
\coordinate (Diplusone) at (4.8138,6.6897);

\coordinate (Wiminusone) at ($(Pimone)!0.5!(Pniminusone)$);
\coordinate (Wi) at ($(Pi)!0.5!(Pni)$);
\coordinate (Wiplusone) at ($(Piplusone)!0.5!(Pniplusone)$);

\draw[diagonal] (Pimone) -- (Pniminusone);
\draw[diagonal] (Pi) -- (Pni);
\draw[diagonal] (Piplusone) -- (Pniplusone);

\draw[polygon] (Pimone) -- (Pi) -- (Piplusone);
\draw[polygon] (Pniplusone) -- (Pni) -- (Pniminusone);

\draw[cssedge]
        (Di) -- (Diplusone);

\draw[wigner] (Wiminusone) -- (Wi) -- (Wiplusone);

\fill[black] (Pimone) circle (2.2pt);
\fill[black] (Pi) circle (2.2pt);
\fill[black] (Piplusone) circle (2.2pt);

\fill[black] (Pniminusone) circle (2.2pt);
\fill[black] (Pni) circle (2.2pt);
\fill[black] (Pniplusone) circle (2.2pt);

\fill[pointblue] (Di) circle (2.2pt);
\fill[pointblue] (Diplusone) circle (2.2pt);

\fill[diagramviolet] (Wiminusone) circle (2.4pt);
\fill[diagramviolet] (Wi) circle (2.4pt);
\fill[diagramviolet] (Wiplusone) circle (2.4pt);

\node[black,below left=2pt] at (Pimone) {$P_{i-1}$};
\node[black,below=3pt] at (Pi) {$P_i$};
\node[black,right=4pt] at (Piplusone) {$P_{i+1}$};

\node[black,above right=2pt] at (Pniminusone) {$P_{n+i-1}$};
\node[black,above left=2pt] at (Pni) {$P_{n+i}$};
\node[black,left=5pt] at (Pniplusone) {$P_{n+i+1}$};

\node[black,above right=0pt] at (Diplusone) {$D_{i+1}$};
\node[black,left=5pt] at (Di) {$D_i$};

\node[black, left=2pt] at (Wiminusone) {$W_{i-1}$};
\node[black,below=3pt, xshift=-2pt] at (Wi) {$W_i$};
\node[black,above=2pt] at (Wiplusone) {$W_{i+1}$};

\node[ anchor=east, font=\large ] at (13.1,6.8) {% 
$\begin{aligned} k_{i+1}\lvert a_{2i}\rvert &= \lvert P_{n+i}D_{n+i}\rvert \\[3pt] \lvert a_{2i}\rvert &= \lvert W_iD_{i+1}\rvert \\[3pt] k_i\lvert a_{2i-1}\rvert &= \lvert P_{n+i}D_i\rvert \\[3pt] \lvert a_{2i-1}\rvert &= \lvert D_iW_i\rvert \end{aligned}$ };
\end{tikzpicture}}
    \caption{The point $W_i$ lies on the segment $P_iD_i$}
    \label{fig:ReconstructionProof1}
\end{figure}
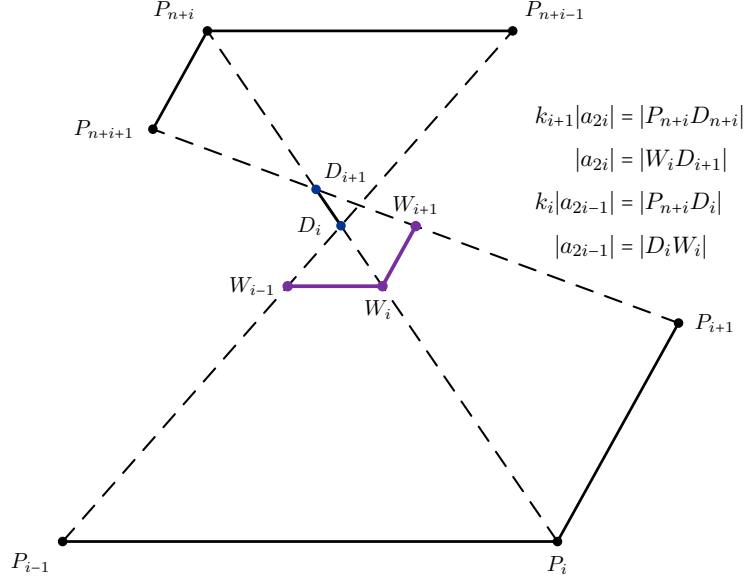

Since $W_{i-1}W_i$ is parallel to $P_{n+i-1}P_{n+i}$, triangle $\triangle P_{n+i-1}P_{n+i}D_i$ is similar to triangle $\triangle W_{i-1}W_iD_i$ with proportionality coefficient $k_i$ (we denote pairs of triangles with vertex $D_i$ as having proportionality coefficient $k_i$, where the edges of the triangles are disjoint, the only point of contact is $D_i$). Then  $$|P_{n+i}D_i|=k_i|D_iW_i|=k_i|a_{2i-1}|.$$ In an analogous way, triangle $\triangle P_{n+i}P_{n+i+1}D_{i+1}$ is similar to triangle \linebreak $\triangle W_iW_{i+1}D_{i+1}$. Therefore, $$|P_{i+n}D_{i+1}|=k_{i+1}|D_{i+1}W_i|=k_{i+1}|a_{2i}|.$$ Hence,
\begin{align*}
(k_i+1)|a_{2i-1}|=|W_iP_{n+i}|=(k_{i+1}+1)|a_{2i}|.
\end{align*}
Let's also note that:
\begin{align*}
|W_{i-1}P_{n+i-1}|=(k_i+1)|a_{2i-2}|\quad\text{and}\quad |W_{i+1}P_{n+i+1}|=(k_{i+1}+1)|a_{2i+1}|.
\end{align*}

\item \textbf{$W_i$ is on the segment $D_iD_{i+1}$} (see Figure \ref{fig:ReconstructionProof2})

Since $W_{i-1}W_i$ is parallel to $P_{i-1}P_i$,  $\triangle P_{i-1}P_iD_i$ is similar to $\triangle W_{i-1}W_iD_i$ with proportionality coefficient $k_i$ (once again, we denote pairs of triangles with vertex $D_i$ as having proportionality coefficient $k_i$ -- it is the same coefficient as in the last case). Then we obtain $$|P_iD_i|=k_i|D_iW_i|=k_i|a_{2i-1}|.$$ Hence, $\triangle P_{n+i}P_{n+i+1}D_{i+1}$ is similar to $\triangle W_iW_{i+1}D_{i+1}$ and  $$|P_{i+n}D_{i+1}|=k_{i-1}|D_{i-1}W_i|=k_{i+1}|a_{2i}|.$$ Therefore,
\begin{align*}
(k_i+1)|a_{2i-1}|=|P_iW_i|=|W_iP_{n+i}|=(k_{i+1}+1)|a_{2i}|,
\end{align*}
where we used the fact that $W_i$ is the midpoint of segment $P_iP_{n+i}$.
Notice that
\begin{align*}
|W_{i-1}P_{i-1}|=(k_i+1)|a_{2i-2}|\quad\text{and}\quad|W_{i+1}P_{n+i+1}|=(k_{i+1}+1)|a_{2i+1}|.
\end{align*}

\begin{figure}[h]
    \centering

    \resizebox{0.85\linewidth}{!}{%
    \begin{tikzpicture}[
        x=0.85cm,
        y=0.85cm,
        line cap=round,
        line join=round,
        polygon/.style={
            draw=black,
            line width=1.25pt
        },
        diagonal/.style={
            draw=black,
            line width=0.95pt,
            dash pattern=on 8pt off 6pt
        },
        cssedge/.style={
            draw=black,
            line width=1.45pt
        },
        wigner/.style={
            draw=diagramviolet,
            line width=1.8pt
        },
        bluelabel/.style={
            text=pointblue,
            font=\large
        },
        violetlabel/.style={
            text=diagramviolet,
            font=\large
        }
    ]

    %------------------------------------------------
    % Vertices of the CPPOS
    %
    %------------------------------------------------
    \coordinate (Pimone) at (1.20,0.80);
    \coordinate (Pi) at (10.50,0.80);
    \coordinate (Piplusone) at (12.34,3.36);

    \coordinate (Pniminusone) at (6.80,9.00);
    \coordinate (Pni) at (2.50,9.00);
    \coordinate (Pniplusone) at (0.20,5.80);

    %------------------------------------------------
    % Intersections of consecutive great diagonals
    %
    %------------------------------------------------
    \coordinate (Di) at (5.0294,6.4074);
    \coordinate (Diplusone) at (6.9444,4.4444);

    %------------------------------------------------
    % Vertices of the Wigner caustic
    %
    %------------------------------------------------
    \coordinate (Wiminusone)
        at ($(Pimone)!0.5!(Pniminusone)$);

    \coordinate (Wi)
        at ($(Pi)!0.5!(Pni)$);

    \coordinate (Wiplusone)
        at ($(Piplusone)!0.5!(Pniplusone)$);

    %------------------------------------------------
    % Great diagonals
    %------------------------------------------------
    \draw[diagonal]
        (Pimone) -- (Pniminusone);

    \draw[diagonal]
        (Pni) -- (Pi);

    \draw[diagonal]
        (Pniplusone) -- (Piplusone);

    %------------------------------------------------
    % Visible edges of the CPPOS
    %------------------------------------------------
    \draw[polygon]
        (Pniplusone) -- (Pni) -- (Pniminusone);

    \draw[polygon]
        (Pimone) -- (Pi) -- (Piplusone);

    %------------------------------------------------
    % The corresponding edge of the centre symmetry set
    %------------------------------------------------
    \draw[cssedge]
        (Di) -- (Diplusone);

    %------------------------------------------------
    % The corresponding fragment of the Wigner caustic
    %------------------------------------------------
    \draw[wigner]
        (Wiminusone) -- (Wi) -- (Wiplusone);

    %------------------------------------------------
    % Blue points: vertices of the polygon and the CSS
    %------------------------------------------------
    \fill[black] (Pimone) circle (2.2pt);
    \fill[black] (Pi) circle (2.2pt);
    \fill[black] (Piplusone) circle (2.2pt);

    \fill[black] (Pniminusone) circle (2.2pt);
    \fill[black] (Pni) circle (2.2pt);
    \fill[black] (Pniplusone) circle (2.2pt);

    \fill[pointblue] (Di) circle (2.2pt);
    \fill[pointblue] (Diplusone) circle (2.2pt);

    %------------------------------------------------
    % Violet points: vertices of the Wigner caustic
    %------------------------------------------------
    \fill[diagramviolet] (Wiminusone) circle (2.4pt);
    \fill[diagramviolet] (Wi) circle (2.4pt);
    \fill[diagramviolet] (Wiplusone) circle (2.4pt);

    %------------------------------------------------
    % Labels of the polygon vertices
    %------------------------------------------------
    \node[
        black,
        below left=3pt
    ] at (Pimone) {$P_{i-1}$};

    \node[
        black,
        below=4pt
    ] at (Pi) {$P_i$};

    \node[
        black,
        right=4pt
    ] at (Piplusone) {$P_{i+1}$};

    \node[
        black,
        above right=2pt
    ] at (Pniminusone) {$P_{n+i-1}$};

    \node[
        black,
        above left=2pt
    ] at (Pni) {$P_{n+i}$};

    \node[
        black,
        left=5pt
    ] at (Pniplusone) {$P_{n+i+1}$};

    %------------------------------------------------
    % Labels of the centre symmetry set vertices
    %------------------------------------------------
    \node[
        black,
        right=3pt
    ] at (Di) {$D_i$};

    \node[
        black,
        right=5pt,
        yshift=4pt
    ] at (Diplusone) {$D_{i+1}$};

    %------------------------------------------------
    % Labels of the Wigner caustic vertices
    %------------------------------------------------
    \node[
        black,
        above=7pt
    ] at (Wiminusone) {$W_{i-1}\ \ \ $};

    \node[
        black,
        above right=2pt
    ] at (Wi) {$W_i$};

    \node[
        black,
        below=5pt
    ] at (Wiplusone) {$W_{i+1}$};

    %------------------------------------------------
    % Length relations displayed next to the diagram
    %------------------------------------------------
    \node[ anchor=east, text=black, font=\large ] at (13.30,8.20) {% 
    $\begin{aligned} k_i\lvert a_{2i-1}\rvert &= \lvert P_{n+i}D_i\rvert \\[4pt] \lvert a_{2i-1}\rvert &= \lvert D_iW_i\rvert \\[4pt] \lvert a_{2i}\rvert &= \lvert W_iD_{i+1}\rvert \\[4pt] k_{i+1}\lvert a_{2i}\rvert &= \lvert D_{i+1}P_i\rvert \end{aligned}$ };

    \end{tikzpicture}%
    }

    \caption{The point $W_i$ lies on the segment $D_iD_{i+1}$}
    \label{fig:ReconstructionProof2}
\end{figure}
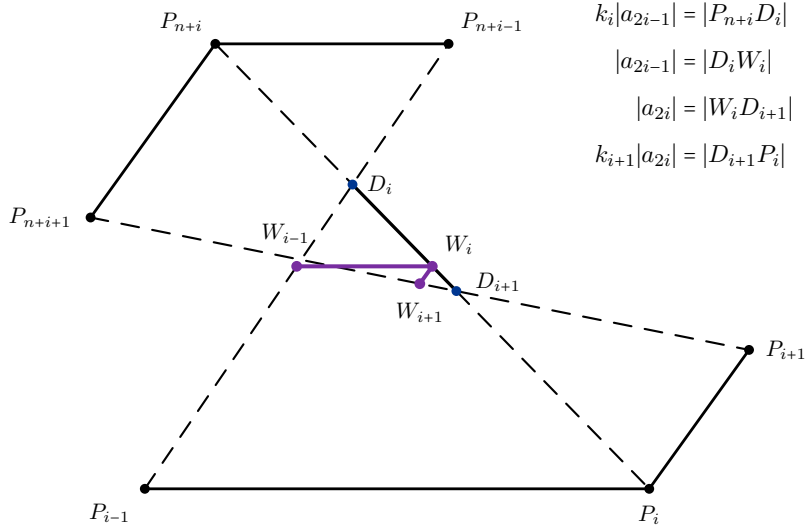

\item \textbf{$W_i$ is on the segment $D_{i+1}P_{n+i}$} (see Figure \ref{fig:ReconstructionProof3})

\begin{figure}[h]
    \centering

    \resizebox{0.85\linewidth}{!}{%
    \begin{tikzpicture}[
        x=0.85cm,
        y=0.85cm,
        line cap=round,
        line join=round,
        polygon/.style={
            draw=black,
            line width=1.25pt
        },
        diagonal/.style={
            draw=black,
            line width=0.95pt,
            dash pattern=on 8pt off 6pt
        },
        cssedge/.style={
            draw=black,
            line width=1.45pt
        },
        wigner/.style={
            draw=diagramviolet,
            line width=1.8pt
        },
        bluelabel/.style={
            text=pointblue,
            font=\large
        },
        violetlabel/.style={
            text=diagramviolet,
            font=\large
        }
    ]

    %------------------------------------------------
    % Vertices of the CPPOS
    %
    % The sides P_{i-1}P_i and
    % P_{n+i-1}P_{n+i} are parallel.
    %
    % The sides P_iP_{i+1} and
    % P_{n+i}P_{n+i+1} are also parallel.
    %------------------------------------------------
    \coordinate (Pimone) at (5.30,0.80);
    \coordinate (Pi) at (10.00,0.80);
    \coordinate (Piplusone) at (12.20,3.00);

    \coordinate (Pniminusone) at (10.70,8.60);
    \coordinate (Pni) at (2.80,8.60);
    \coordinate (Pniplusone) at (0.05,5.85);

    %------------------------------------------------
    % Vertices of the centre symmetry set
    %
    % D_i is the intersection of the lines
    % P_{i-1}P_{n+i-1} and P_iP_{n+i}.
    %
    % D_{i+1} is the intersection of the lines
    % P_iP_{n+i} and P_{i+1}P_{n+i+1}.
    %------------------------------------------------
    \coordinate (Di) at (7.3143,3.7095);
    \coordinate (Diplusone) at (6.8000,4.2667);

    %------------------------------------------------
    % Vertices of the Wigner caustic
    %
    % Each W_j is the midpoint of the corresponding
    % great diagonal P_jP_{n+j}.
    %------------------------------------------------
    \coordinate (Wiminusone)
        at ($(Pimone)!0.5!(Pniminusone)$);

    \coordinate (Wi)
        at ($(Pi)!0.5!(Pni)$);

    \coordinate (Wiplusone)
        at ($(Piplusone)!0.5!(Pniplusone)$);

    %------------------------------------------------
    % Great diagonals
    %------------------------------------------------
    \draw[diagonal]
        (Pimone) -- (Pniminusone);

    \draw[diagonal]
        (Pi) -- (Pni);

    \draw[diagonal]
        (Piplusone) -- (Pniplusone);

    %------------------------------------------------
    % Visible edges of the CPPOS
    %------------------------------------------------
    \draw[polygon]
        (Pniplusone) -- (Pni) -- (Pniminusone);

    \draw[polygon]
        (Pimone) -- (Pi) -- (Piplusone);

    %------------------------------------------------
    % The corresponding edge of the centre symmetry set
    %------------------------------------------------
    \draw[cssedge]
        (Diplusone) -- (Di);

    %------------------------------------------------
    % The corresponding fragment of the Wigner caustic
    %------------------------------------------------
    \draw[wigner]
        (Wiminusone) -- (Wi) -- (Wiplusone);

    %------------------------------------------------
    % Blue points: vertices of the CPPOS and the CSS
    %------------------------------------------------
    \fill[black] (Pimone) circle (2.2pt);
    \fill[black] (Pi) circle (2.2pt);
    \fill[black] (Piplusone) circle (2.2pt);

    \fill[black] (Pniminusone) circle (2.2pt);
    \fill[black] (Pni) circle (2.2pt);
    \fill[black] (Pniplusone) circle (2.2pt);

    \fill[pointblue] (Di) circle (2.2pt);
    \fill[pointblue] (Diplusone) circle (2.2pt);

    %------------------------------------------------
    % Violet points: vertices of the Wigner caustic
    %------------------------------------------------
    \fill[diagramviolet] (Wiminusone) circle (2.4pt);
    \fill[diagramviolet] (Wi) circle (2.4pt);
    \fill[diagramviolet] (Wiplusone) circle (2.4pt);

    %------------------------------------------------
    % Labels of the CPPOS vertices
    %------------------------------------------------
    \node[
        black,
        below left=3pt
    ] at (Pimone) {$P_{i-1}$};

    \node[
        black,
        below=4pt
    ] at (Pi) {$P_i$};

    \node[
        black,
        right=4pt
    ] at (Piplusone) {$P_{i+1}$};

    \node[
        black,
        above right=2pt
    ] at (Pniminusone) {$P_{n+i-1}$};

    \node[
        black,
        above left=2pt
    ] at (Pni) {$P_{n+i}$};

    \node[
        black,
        left=5pt
    ] at (Pniplusone) {$P_{n+i+1}$};

    %------------------------------------------------
    % Labels of the centre symmetry set vertices
    %------------------------------------------------
    \node[
        black,
        right=2pt
    ] at (Di) {$D_i$};

    \node[
        black,
        below=2pt
    ] at (Diplusone) {$D_{i+1}\ \ $};

    %------------------------------------------------
    % Labels of the Wigner caustic vertices
    %------------------------------------------------
    \node[
        black,
        right=2pt
    ] at (Wiminusone) {$W_{i-1}$};

    \node[
        black,
        above=3pt
    ] at (Wi) {$W_i$};

    \node[
        black,
        below left=2pt
    ] at (Wiplusone) {$W_{i+1}$};

    %------------------------------------------------
    % Length relations displayed next to the diagram
    %------------------------------------------------
    \node[ anchor=east, font=\large ] at (13.10,5.40) {% 
    $\begin{aligned} \lvert a_{2i-1}\rvert &= \lvert W_iD_i\rvert \\[3pt] k_i\lvert a_{2i-1}\rvert &= \lvert D_iP_i\rvert \\[3pt] \lvert a_{2i}\rvert &= \lvert W_iD_{i+1}\rvert \\[3pt] k_{i+1}\lvert a_{2i}\rvert &= \lvert D_{i+1}P_i\rvert \end{aligned}$ };
    \end{tikzpicture}%
    }

    \caption{The point $W_i$ lies on the segment $D_{i+1}P_{n+i}$}
    \label{fig:ReconstructionProof3}
\end{figure}
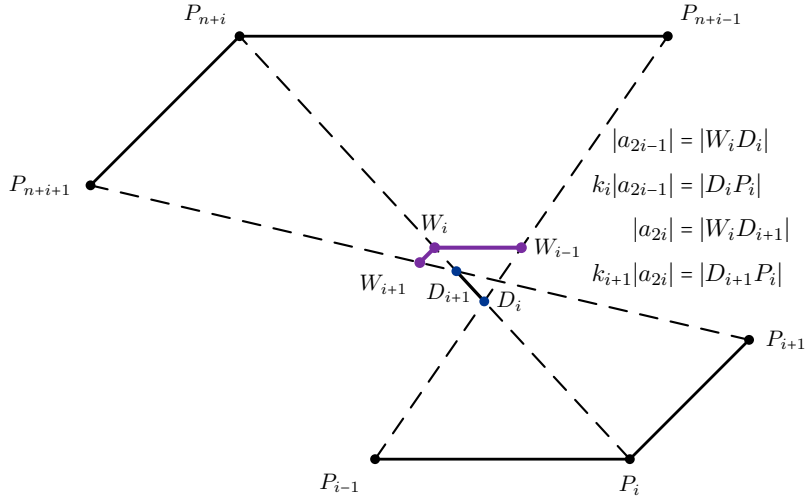

This case is analogous to the first one -- we only swap the indices of $P_{i-1}$, $P_i$, $P_{i+1}$ and their counterparts $P_{n+i-1}$, $P_{n+i}$, $P_{n+i+1}$, respectively. The rest of the proof is identical. So one more time we obtain: 
\begin{align*}
(k_i+1)|a_{2i-1}|=|W_iP_i|=(k_{i+1}+1)|a_{2i}|.
\end{align*}
Again,
\begin{align*}
|W_{i-1}P_{i-1}|=(k_i+1)|a_{2i-2}|\quad\text{and}\quad|W_{i+1}P_{i+1}|=(k_{i+1}+1)|a_{2i+1}|.
\end{align*}
\end{enumerate}

Since $|P_iW_i|=|W_iP_{n+i}|$ for all $i \in\{1,2,\ldots,n\}$ and by equalities we obtained in the three considered cases above, we get the following system of equations:
\begin{align}
\label{eq:systemeq1}
\left\{
\begin{aligned}
(k_1+1)|a_1| &= (k_2+1)|a_2|, \\
(k_2+1)|a_3| &= (k_3+1)|a_4|, \\
             &\vdots \\
(k_{n-1}+1)|a_{2n-3}| &= (k_n+1)|a_{2n-2}|, \\
(k_n+1)|a_{2n-1}| &= (k_1+1)|a_{2n}|.
\end{aligned}
\right.
\end{align}
By multiplying the sides in \eqref{eq:systemeq1}, we easily obtain the necessary condition
\begin{align*}
|a_1|\cdot |a_3|\cdot\ldots\cdot |a_{2n-1}| = |a_2|\cdot |a_4|\cdot\ldots\cdot |a_{2n}|
\end{align*} 
that is required for all closed polygonal chains with $2n$ vertices. 

Now we assume that a certain Wigner caustic and centre symmetry set meet the condition \eqref{eq:reconstruction}. We will recreate the polygonal chain using the reconstruction procedure that was described earlier, where $P_1$ is the starting point, and we stop before drawing the last two sides of the polygonal chain. The reconstruction is possible if and only if $P_{n-1}P_n$ or $P_{2n-1}P_{2n}$ is parallel to $W_{n-1}W_n$, because one parallelism implies the second one from triangle similarity and the fact that $W_{n-1}$, $W_n$ are the midpoints of the segments $P_{n-1}P_{2n-1}, P_nP_{2n}$ (see Figure \ref{fig:reconstructionProof4}). Without loss of generality, let's consider the case where $D_n$ is a common vertex of triangles $\triangle P_{n-1}P_nD_n$ and $\triangle W_{n-1}W_nD_n$, which are disjoint except for $D_n$. 

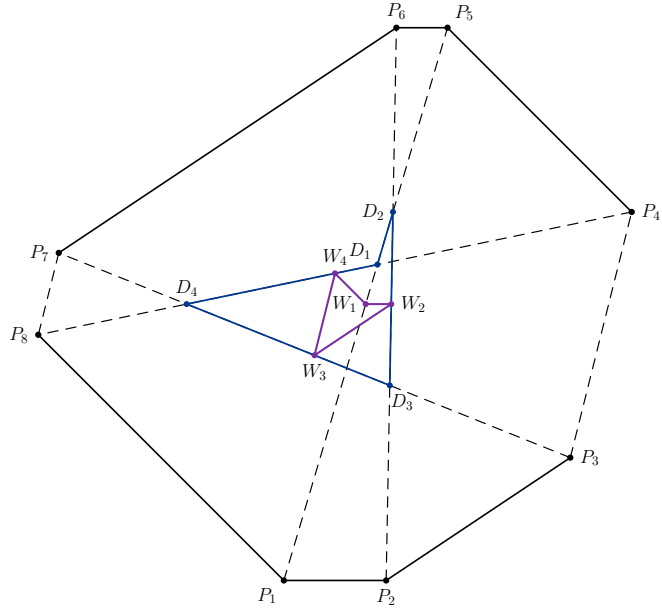
\begin{figure}[h]
    \centering

    \resizebox{0.69\linewidth}{!}{%
    \begin{tikzpicture}[
        x=0.55cm,
        y=0.55cm,
        line cap=round,
        line join=round,
        polygon/.style={
            draw=black,
            line width=1.25pt
        },
        diagonal/.style={
            draw=black,
            line width=0.9pt,
            dash pattern=on 8pt off 6pt
        },
        css/.style={
            draw=pointblue,
            line width=1.35pt
        },
        wigner/.style={
            draw=diagramviolet,
            line width=1.55pt
        },
        Plabel/.style={
            text=black,
            font=\large
        },
        Dlabel/.style={
            text=pointblue,
            font=\large
        },
        Wlabel/.style={
            text=diagramviolet,
            font=\large
        }
    ]

    \coordinate (P1) at (0,0);
    \coordinate (P2) at (5,0);
    \coordinate (P3) at (14,6);
    \coordinate (P4) at (17,18);

    \coordinate (P5) at (8,27);
    \coordinate (P6) at (5.5,27);
    \coordinate (P7) at (-11,16);
    \coordinate (P8) at (-12,12);

    \coordinate (W1) at ($(P1)!0.5!(P5)$);
    \coordinate (W2) at ($(P2)!0.5!(P6)$);
    \coordinate (W3) at ($(P3)!0.5!(P7)$);
    \coordinate (W4) at ($(P4)!0.5!(P8)$);

    \coordinate (D1) at (4.5714286,15.4285714);
    \coordinate (D2) at (5.3333333,18);
    \coordinate (D3) at (5.1764706,9.5294118);
    \coordinate (D4) at (-4.75,13.5);

    %------------------------------------------------
    % The outer CPPOS
    %------------------------------------------------
    \draw[polygon]
%        (P1) -- (P2) -- (P3) -- (P4) --
%        (P5) -- (P6) -- (P7) -- (P8) -- cycle;
        (P8)--(P1)--(P2)--(P3);
    \draw[polygon] (P4)--(P5)--(P6)--(P7);
    \draw[diagonal] (P7)--(P8);
    \draw[diagonal] (P3)--(P4);

    %------------------------------------------------
    % Great diagonals
    %------------------------------------------------
    \draw[diagonal] (P1) -- (P5);
    \draw[diagonal] (P2) -- (P6);
    \draw[diagonal] (P3) -- (P7);
    \draw[diagonal] (P4) -- (P8);

    %------------------------------------------------
    % Centre symmetry set
    %------------------------------------------------
    \draw[css]
        (D1) -- (D2) -- (D3) -- (D4) -- cycle;

    %------------------------------------------------
    % Wigner caustic
    %------------------------------------------------
    \draw[wigner]
        (W1) -- (W2) -- (W3) -- (W4) -- cycle;

    %------------------------------------------------
    % Vertices of the CPPOS
    %------------------------------------------------
    \fill[black] (P1) circle (2.4pt);
    \fill[black] (P2) circle (2.4pt);
    \fill[black] (P3) circle (2.4pt);
    \fill[black] (P4) circle (2.4pt);
    \fill[black] (P5) circle (2.4pt);
    \fill[black] (P6) circle (2.4pt);
    \fill[black] (P7) circle (2.4pt);
    \fill[black] (P8) circle (2.4pt);

    %------------------------------------------------
    % Vertices of the centre symmetry set
    %------------------------------------------------
    \fill[pointblue] (D1) circle (2.3pt);
    \fill[pointblue] (D2) circle (2.3pt);
    \fill[pointblue] (D3) circle (2.3pt);
    \fill[pointblue] (D4) circle (2.3pt);

    %------------------------------------------------
    % Vertices of the Wigner caustic
    %------------------------------------------------
    \fill[diagramviolet] (W1) circle (2.3pt);
    \fill[diagramviolet] (W2) circle (2.3pt);
    \fill[diagramviolet] (W3) circle (2.3pt);
    \fill[diagramviolet] (W4) circle (2.3pt);

    %------------------------------------------------
    % Labels of the CPPOS vertices
    %------------------------------------------------
    \node[Plabel,below left=2pt]
        at (P1) {\LARGE $P_1$};

    \node[Plabel,below=3pt]
        at (P2) {\LARGE $P_2$};

    \node[Plabel,right=4pt]
        at (P3) {\LARGE $P_3$};

    \node[Plabel,right=4pt]
        at (P4) {\LARGE $P_4$};

    \node[Plabel,above right=2pt]
        at (P5) {\LARGE $P_5$};

    \node[Plabel,above=3pt]
        at (P6) {\LARGE $P_6$};

    \node[Plabel,left=4pt]
        at (P7) {\LARGE $P_7$};

    \node[Plabel,left=4pt]
        at (P8) {\LARGE $P_8$};

    %\node[Plabel,above=24pt]
    %    at (P8) {\Huge$?$};

    %\node[Plabel,below=80pt]
    %    at (P4) {\Huge$?$};
    \node[Plabel,above left=1pt]
        at (D1) {\LARGE $D_1$};

    \node[Plabel,left=3pt]
        at (D2) {\LARGE $D_2$};

    \node[Plabel,below=3pt]
        at (D3) {\LARGE $\ \ \ \ D_3$};

    \node[Plabel,above=3pt]
        at (D4) {\LARGE $D_4$};
   \node[Plabel,left=3pt]
        at (W1) {\LARGE $W_1$};

    \node[Plabel, right=4pt]
        at (W2) {\LARGE $W_2$};

    \node[Plabel,below=3pt]
        at (W3) {\LARGE $W_3$};

    \node[Plabel,above=3pt]
        at (W4) {\LARGE $W_4$};

    \end{tikzpicture}%
    }
    \caption{Case for $n=4$}
    \label{fig:reconstructionProof4}
\end{figure}

This condition is satisfied when the triangles mentioned above are similar. Therefore, we need to show the following equation to end the proof of the theorem:
\begin{align}
\label{eq:reconstructionProofEnd}\frac{|P_nD_n|}{|P_{n-1}D_n|}=\frac{|a_{2n-1}|}{|a_{2n-2}|}.
\end{align}
Also, from triangle similarities, using the reconstruction method, we get almost identical systems of equations in an analogous way as before (of course, without equations containing $k_n$):
\begin{align}
\label{eq:systemofeq2}
\left\{
\begin{aligned}
(k_1+1)|a_1| &= (k_2+1)|a_2|, \\
(k_2+1)|a_3| &= (k_3+1)|a_4|, \\
             &\vdots \\
(k_{n-2}+1)|a_{2n-5}| &= (k_{n-1}+1)|a_{2n-4}|.
\end{aligned}
\right.
\end{align}
Since
\[
(k_{n-1}+1)|a_{2n-3}| = |a_{2n-2}|+|P_{n-1}D_n| \text{ and }
(k_1+1)|a_{2n}| = |a_{2n-1}|+|P_nD_n|,
\]
we obtain that
\[
\frac{|P_nD_n|}{|P_{n-1}D_n|}=\frac{(k_1+1)|a_{2n}|-|a_{2n-1}|}{(k_{n-1}+1)|a_{2n-3}|-|a_{2n-2}|}.
\]
Then, by \eqref{eq:systemofeq2}, we get the following equations:
\begin{align*}
(k_1+1)|a_{2n}|&=\frac{|a_{2n}|}{|a_1|}(k_1+1)|a_1|=\frac{|a_{2n}|}{|a_1|}(k_2+1)|a_2|=\frac{|a_{2n}||a_2|}{|a_1||a_3|}(k_2+1)|a_3|=\dots\\ 
&
=(k_{n-2}+1)|a_{2n-5}|\cdot\frac{|a_{2n}||a_2|\dots|a_{2n-6}|}{|a_1|\dots|a_{2n-5}|}=(k_{n-1}+1)|a_{2n}|\cdot\frac{|a_2|\dots|a_{2n-4}|}{|a_1|\dots|a_{2n-5}|}.
\end{align*}
Finally, using the fact that $|a_1|\cdot |a_3|\cdot\ldots\cdot |a_{2n-1}| = |a_2|\cdot |a_4|\cdot\ldots\cdot |a_{2n}|$, we get:
\begin{align*}
\frac{|P_nD_n|}{|P_{n-1}D_n|}
&=\frac{(k_{n-1}+1)|a_{2n}|\cdot\frac{|a_2|\dots|a_{2n-4}|}{|a_1|\dots|a_{2n-5}|}-|a_{2n-1}|}{(k_{n-1}+1)|a_{2n-3}|-|a_{2n-2}|}
\\ &=
\frac{(k_{n-1}+1)|a_{2n}|\cdot\frac{|a_2|\dots|a_{2n-4}|}{|a_1|\dots|a_{2n-5}|}-|a_{2n-1}|}{(k_{n-1}+1)|a_{2n-3}|-|a_{2n-2}|}\cdot\frac{\frac{|a_{2n-2}|}{|a_{2n-3}||a_{2n-1}|}}{\frac{|a_{2n-2}|}{|a_{2n-3}||a_{2n-1}|}}
\\ &=\frac{(k_{n-1}+1)-\frac{|a_{2n-2}|}{|a_{2n-3}|}}{\frac{(k_{n-1}+1)|a_{2n-2}|}{|a_{2n-1}|}-\frac{|a_{2n-2}|^2}{|a_{2n-3}||a_{2n-1}|}}
=
\frac{|a_{2n-1}|}{|a_{2n-2}|}\cdot\frac{(k_{n-1}+1)-\frac{|a_{2n-2}|}{|a_{2n-3}|}}{(k_{n-1}+1)-\frac{|a_{2n-2}|}{|a_{2n-3}|}}
=\frac{|a_{2n-1}|}{|a_{2n-2}|}.
\end{align*}
Therefore, equation \eqref{eq:reconstructionProofEnd} holds, and we finish the proof.
\end{proof}

Theorem \ref{thmReconstruction} provides a general criterion for reconstructing arbitrary closed polygonal chains. However, in order to reconstruct a $\cppos$, additional information is required. Figure~\ref{chains} shows examples of closed polygonal chains reconstructed solely from the potential Wigner caustic and the centre symmetry set that satisfy equation~\eqref{eq:reconstruction}.

We now present a theorem that allows us to determine whether a $\cppos$ can be reconstructed from a given potential Wigner caustic. From now on, we consider only potential Wigner caustics that indeed satisfy the necessary conditions for being a Wigner caustic, that is, they have an odd number of cusps greater than or equal to three.

\begin{figure}[h!]
\centering
\begin{subfigure}[b]{0.49\textwidth}
\centering
\includegraphics[height=4.5cm,keepaspectratio]{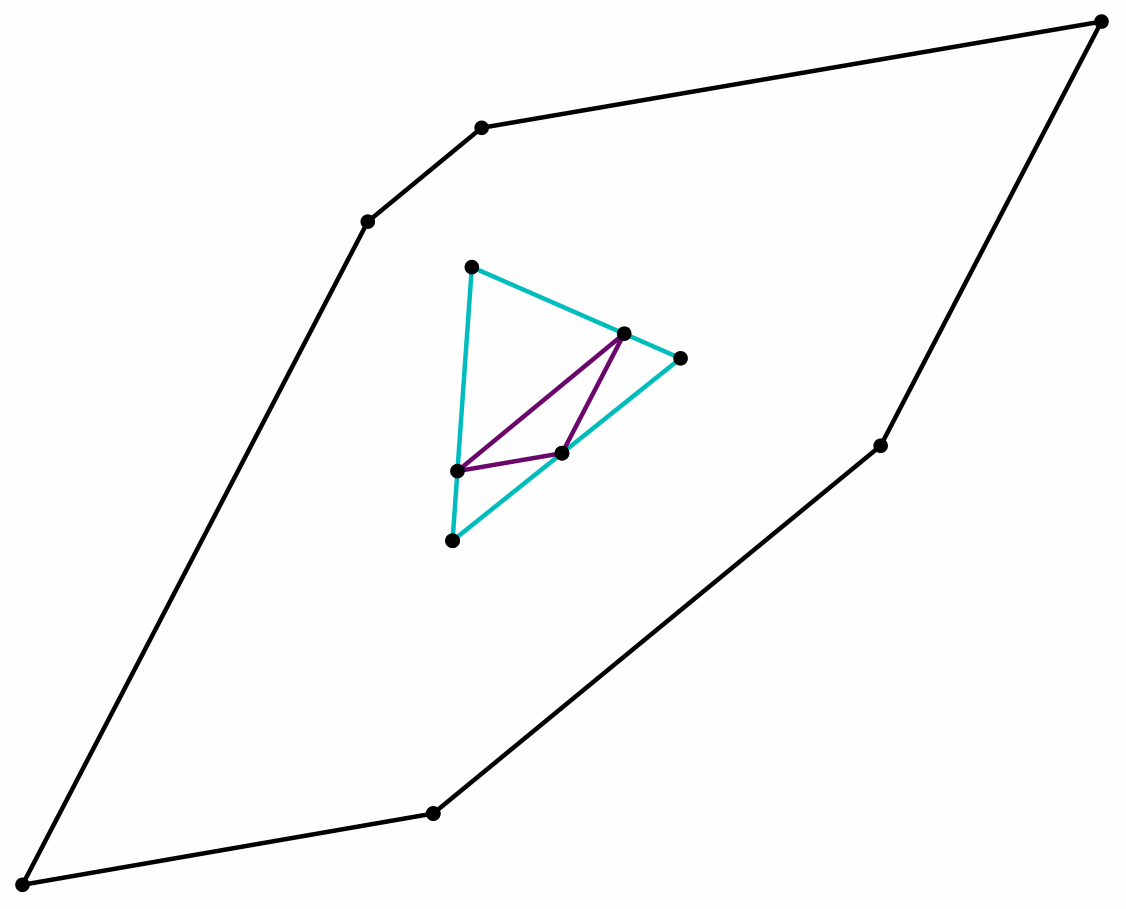}
\subcaption{A closed polygonal chain with $6$ vertices}
\end{subfigure}
\hfill
\begin{subfigure}[b]{0.49\textwidth}
\centering
\includegraphics[height=4.0cm,keepaspectratio]{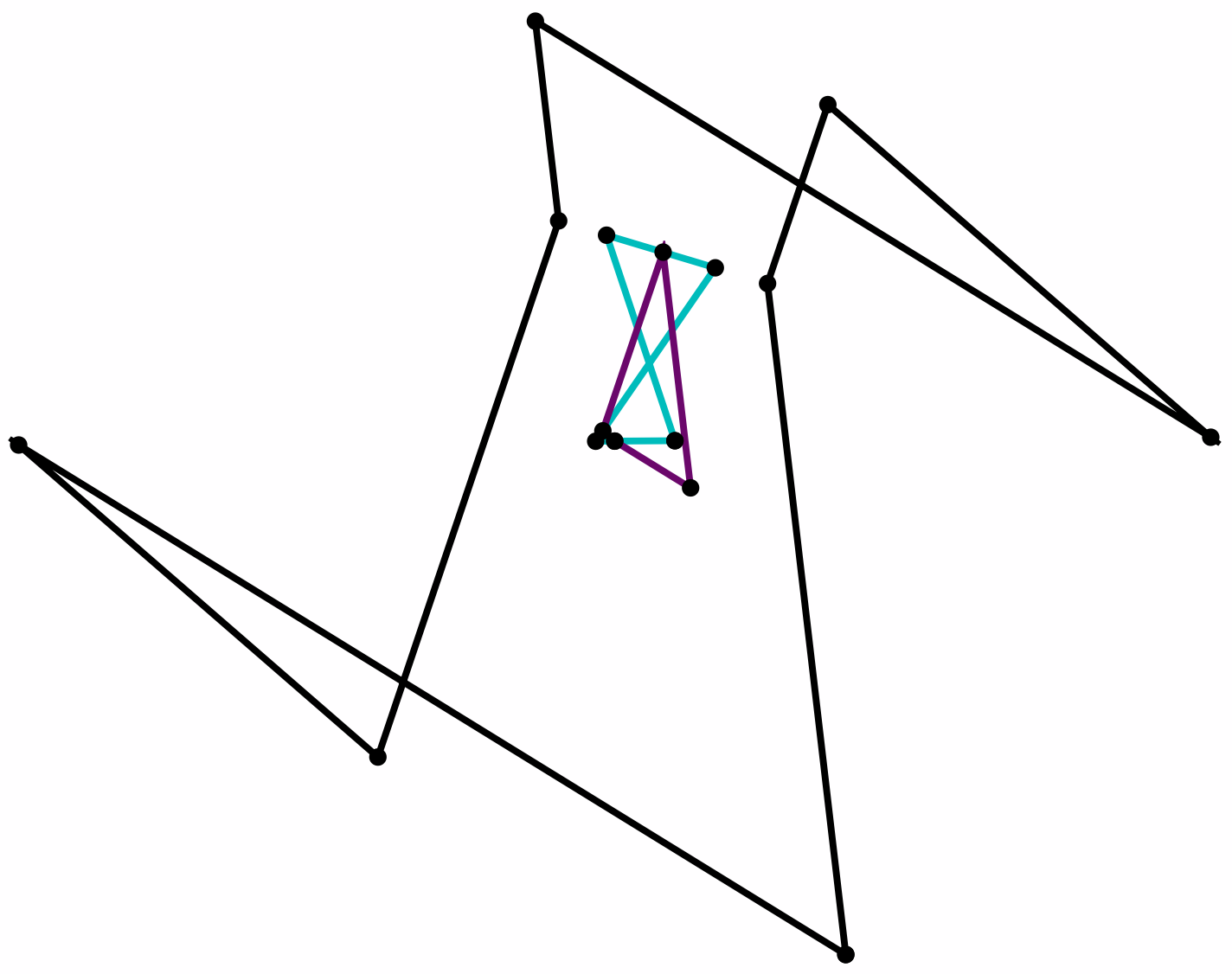}
\subcaption{A closed polygonal chain with $8$ vertices}
\end{subfigure}
\caption{Closed polygonal chains reconstructed from singular sets}
\label{chains}
\end{figure}

\pagebreak

Consider a potential Wigner caustic and extend all of its edges. Each pair of consecutive extended edges determines two pairs of adjacent angles. Denote by $\alpha_i$ the exterior angles at the $i$-th vertex of the potential Wigner caustic, together with the angles equal in measure to them, as shown in Figure~\ref{katy}.

\begin{figure}[h!]
\centering
\includegraphics[width=1\linewidth]{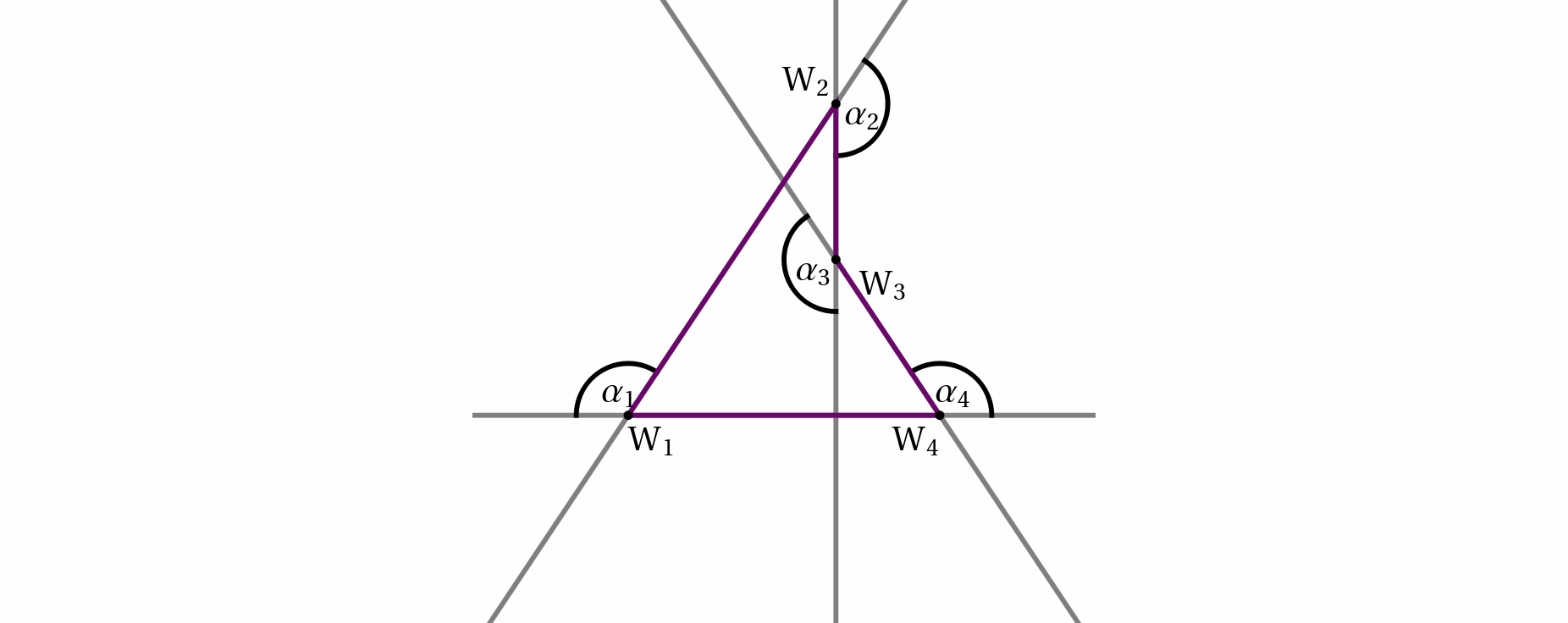}
\caption{Exterior angles of a potential Wigner caustic}
\label{katy}
\end{figure}

We slightly modify the standard definition of an exterior angle. For convenience, we define the exterior angle corresponding to a reflex interior angle as the angle that complements it to $2\pi$. As usual, the exterior angle corresponding to an interior angle of measure less than $\pi$ is defined as the angle that complements it to $\pi$.

In what follows, by the existence of a reconstruction of a $\cppos$ we mean the~existence of a $\cppos$ within the one-parameter family obtained by varying the~initial point of the reconstruction.

\begin{thm}\label{thm:angleRecon}
Assume that equality~\eqref{eq:reconstruction} holds.
The reconstruction of a $\cppos$ exists if and only if
\begin{align}
\label{eq:anglesSum}
\sum_{i=1}^{n} 2\alpha_i = (2n-2)\cdot \pi,
\end{align}
where the angles $\alpha_i$ are the exterior angles of the potential Wigner caustic.
\end{thm}

\begin{proof}
The angles $\alpha_i$ are not only exterior angles of the potential Wigner caustic but also angles of the reconstructed closed polygonal chain. This follows from the~fact that the sides of the reconstructed polygon are parallel to the sides of the~Wigner caustic.

A $\cppos$ is a simple polygon with $2n$ vertices, and therefore, the sum of its interior angles must be equal to $(2n-2)\cdot\pi$. Consequently, suppose that equality~\eqref{eq:anglesSum} holds. By varying the initial point in the reconstruction procedure, we obtain a one-parameter family of closed polygonal chains, whose members are affine equidistants of one another. The initial point can be chosen so that the consecutive edges follow the cyclic order of their prescribed directions. For this choice, all turning angles have the same orientation and, by \eqref{eq:anglesSum}, their total turning is $2\pi$. Hence, the resulting polygonal chain is a simple convex polygon. Since its corresponding opposite sides are parallel by construction, it is a CPPOS.
\end{proof}

\begin{remark}
A point belonging to the polygon that is being reconstructed is necessary for a unique reconstruction, since the centre symmetry set and the Wigner caustic are the same for the entire family of affine equidistants (see equations \eqref{eq:wcthesame} and \eqref{eq:cssthesame}).
\end{remark}

\begin{remark}
Recall that the centre symmetry set is formed by singular
points of the family of affine equidistants. Consequently,
the reconstruction yields a $\cppos$ if and only if the
conditions stated in Theorem~\ref{thm:angleRecon} hold and, in addition,
the~reconstructed polygon and its centre symmetry set
have no common points.
\end{remark}

Thus, we have successfully completed the procedure of determining the necessary conditions for the reconstruction of convex polygons with parallel corresponding sides from their Wigner caustic and centre symmetry set. 

%%%%%%%%%%%%%%%%%%%%%%%%%%%%%%%%%%%%%%%%%%%%%%%%%%%%%
%%%%%%%%%%%%%%%%%%%%%%%%%%%%%%%%%%%%%%%%%%%%%%%%%%%%%%
%%%%%%%%%%%%%%%%%%%%%%%%%%%%%%%%%%%%%%%%%%%%%%%%%%%%%%

\section{Isoperimetric-Type Inequalities}

\noindent
We begin this section by recalling the classical isoperimetric inequality for polygons. If $\mathcal{P}$ is an $m$-sided convex polygon, then
\begin{align}
\label{eq:classIsoIneqPolygons}
L(\mathcal{P})^2\geqslant 4m\tan\left(\frac{\pi}{m}\right) A(\mathcal{P}),
\end{align}
where $L(\mathcal{P})$ denotes the perimeter of $\mathcal{P}$ and $A(\mathcal{P})$ denotes the area of the region enclosed by $\mathcal{P}$. Moreover, equality holds if and only if $\mathcal{P}$ is a regular polygon. Passing to the limit $m\to\infty$, one recovers the classical isoperimetric inequality for smooth planar curves.

Isoperimetric inequalities for planar objects have attracted considerable attention in recent years, leading to a wide range of new results for smooth curves, convex bodies, and various generalized geometric structures (see \cite{bogosel, Cufi, CentroidIneq, PanShu, isoBook, Safarewicz, SC1, Zhang1, Zwierz1, Zwierz3, Zwierz2, ZwierzMix}). In this context, discrete models provide more than merely finite-dimensional analogues of smooth problems. They may reveal additional combinatorial and geometric features that are difficult to detect in the continuous setting, while at the same time offering effective tools for numerical experiments and for testing possible conjectures. Thus, the discretization considered here may contribute not only to a better understanding of the classical smooth theory, but also to the discovery of new properties and inequalities that may subsequently prove useful in the study of planar geometric objects.

In what follows, we prove a refinement of this inequality for convex polygons with parallel opposite sides. The additional term, which appears in the estimate, is expressed in terms of the oriented area of the Wigner caustic. Thus, in the class of $\cppos$es, the isoperimetric deficit is controlled not only by the area enclosed by the polygon, but also by the affine geometry encoded in its equidistant at level~$0.5$. This gives a discrete analogue of the improved isoperimetric inequality known for smooth ovals and shows that the Wigner caustic naturally measures part of the deviation from the extremal case.

The proofs of Theorem~\ref{thm:discreteImpIsoIneq} and Theorem~\ref{thmMaxAreaWC} make use of another natural construction associated with a convex body $K$, which is its central symmetral: \linebreak
$\frac{1}{2}\bigl(K+(-K)\bigr)$,
while the unnormalized body $K+(-K)=K-K$ is usually called the~difference body of $K$. It is also an example of a secant caustic (see \cite{Romero}). The~central symmetral is an origin-symmetric convex body whose support function is equal to one half of the width function of $K$. Consequently, it has the same width function as $K$, and, in the planar case, the same perimeter. On the other hand, the Brunn--Minkowski and Rogers--Shephard inequalities imply
\begin{align}
\label{BrunnMinkRogShepIneq} A(K)\leqslant A\left(\frac{K-K}{2}\right)\leqslant \frac{3}{2}A(K).
\end{align}
Thus, Minkowski symmetrization preserves the perimeter while increasing the enclosed area, which makes it particularly natural in the study of isoperimetric inequalities and geometric measures of asymmetry (see, for instance, \cite{CentroidIneq, SchneiderBook, ZwierzMix}).

\begin{definition}
We call a $\cppos$ a \textit{constant width $\cppos$} (or a $\cppos$ of constant width) if the distance between every pair of opposite sides is the same. 

Notice that such a polygon need not be a body of constant width in the usual sense. Nevertheless, by the construction described in the second section (see \linebreak Remark \ref{remConstruction}), constant width $\cppos$es arise as polygonal approximations of \linebreak bodies of constant width.
\end{definition}

\begin{definition}
Let $\mathcal{P}=(P_1,\ldots,P_m)$ be a closed oriented polygonal chain in $\mathbb{R}^2$, where
$P_i=(x_i,y_i)$ for $i=1,\ldots,m$,
and let $P_{m+1}=P_1$. The \emph{oriented area} of $\mathcal{P}$ is defined by
\begin{equation}
\label{eq:oriented-area-polygon}
A^\ast(\mathcal{P})
=
\frac{1}{2}
\sum_{i=1}^{m}
\left(
x_i y_{i+1}-x_{i+1}y_i\right)
=
\frac{1}{2}
\sum_{i=1}^{m}
\det(P_i,P_{i+1}).
\end{equation}
Formula~\eqref{eq:oriented-area-polygon} is commonly known as the
\emph{shoelace formula}.
\end{definition}

The sign of $A^\ast(\mathcal{P})$ depends on the orientation of the
polygonal chain. If $\mathcal{P}$ is a simple polygon oriented
counterclockwise, then
$A^\ast(\mathcal{P})=A(\mathcal{P}),$ whereas, for the clockwise orientation,
$A^\ast(\mathcal{P})=-A(\mathcal{P})$, where $A(\mathcal{P})$ denotes the area enclosed by~$\mathcal{P}$. For a self-intersecting polygonal chain, $A^\ast(\mathcal{P})$ is the
signed area counted with multiplicity according to the winding number.

\begin{remark}
    The proof of Theorem \ref{thm:convergence} yields convergence with respect to the~natural cyclic parametrizations. On every compact regular subarc, this convergence is of class $C^1$, while the parts contained in sufficiently small neighbourhoods of the~cusps have arbitrarily small length. Consequently, with the same assumptions as in Theorem \ref{thm:convergence},
\begin{align*}
    L\big(\ppos_{\nu_k}(\Eq_\lambda(M))\big)\to L\big(\Eq_{\lambda}(M)\big)
    \quad\text{and}\quad
    A^\ast\big(\ppos_{\nu_k}(\Eq_\lambda(M))\big)\to A^\ast\big(\Eq_{\lambda}(M)\big).
\end{align*}
The same holds for the oriented area of the centre symmetry set.
\end{remark}

\begin{thm}\label{thm:discreteImpIsoIneq}
Let \(\mathcal P\) be a $\cppos$ with $2n$ vertices. Then
\[
L(\mathcal P)^2
\geqslant
8n\tan\left(\frac{\pi}{2n}\right)\cdot 
\big(
A(\mathcal P)
+
2\left|A^\ast\left(E_{0.5}(\mathcal P)\right)\right|
\big).
\]
Moreover, equality holds if and only if \(\mathcal P\) is a $\cppos$ of constant width and has equally spaced side directions.
\end{thm}

\begin{proof}
Let
$W_i:=P_i(0.5)=\frac{P_i+P_{i+n}}2$
be the vertices of the Wigner caustic. We also define an auxiliary polygon \(\mathcal Q\) by the vertices
$Q_i:=\frac{P_i-P_{i+n}}2$. Thus, $Q$ is the boundary of the central symmetral $\frac{1}{2}(K-K)$ of the convex body $K$ bounded by $\mathcal P$.
Then
$P_i=W_i+Q_i$.

Denote the edges of \(\mathcal Q\) by
$q_i:=Q_{i+1}-Q_i$.
Then
\begin{equation}
q_i
=
\frac{P_{i+1}-P_{i+n+1}-P_i+P_{i+n}}2
=
\frac{e_i-e_{i+n}}2.
\label{eq:Q-edge-first}
\end{equation}
Since \(\mathcal P\) is a $\cppos$, the opposite edges \(e_i\) and \(e_{i+n}\) are parallel and oppositely oriented. Hence, there exists \(\alpha_i>0\) such that
\[
e_{i+n}=-\alpha_i e_i.
\]
Therefore, by \eqref{eq:Q-edge-first},
\begin{equation}
q_i=\frac{1+\alpha_i}{2}e_i.
\label{eq:Q-edge-positive-multiple}
\end{equation}
Thus, every edge of \(\mathcal Q\) is a positive multiple of the corresponding edge of \(\mathcal P\). Since \(\mathcal P\) is convex, it follows that \(\mathcal Q\) is convex as well.

Moreover, from \eqref{eq:Q-edge-positive-multiple} we get
\[
|q_i|=\frac{|e_i|+|e_{i+n}|}{2}.
\]
Hence,
\begin{equation}
L(\mathcal Q)
=
\sum_{i=1}^{2n}|q_i|
=
\frac12\sum_{i=1}^{2n}|e_i|
+
\frac12\sum_{i=1}^{2n}|e_{i+n}|
=
L(\mathcal P).
\label{eq:Q-perimeter}
\end{equation}

We now compare the oriented areas. Since $P_i=W_i+Q_i$, we get
\begin{align*}
A^*(\mathcal P)
&=
\frac12\sum_{i=1}^{2n}
[W_i+Q_i,W_{i+1}+Q_{i+1}] \\
&=
\frac12\sum_{i=1}^{2n}
\bigl([W_i,W_{i+1}]+[W_i,Q_{i+1}]\bigr)
+
\frac12\sum_{i=1}^{2n}
\bigl([Q_i,W_{i+1}]+[Q_i,Q_{i+1}]\bigr) \\
&=
\frac12\sum_{i=1}^{2n}[W_i,W_{i+1}]
+
\frac12\sum_{i=1}^{2n}[Q_i,Q_{i+1}]
+
\frac12\sum_{i=1}^{2n}
\bigl([W_i,Q_{i+1}]+[Q_i,W_{i+1}]\bigr).
\end{align*}

We claim that the mixed terms vanish. Indeed,
\[
W_{i+n}=W_i
\quad \text{and} \quad
Q_{i+n}=-Q_i,
\]
so the mixed term corresponding to \(i+n\) is the negative of the mixed term
corresponding to \(i\). Thus, 
\begin{equation}\label{eq:area-P-W-Q}
A^*(\mathcal P)=A^*(W^{(2)})+A^*(\mathcal Q),
\end{equation}
where $W^{(2)}$ denotes the polygonal chain $W_1,W_2,\ldots,W_{2n}$.

Since \(W_{i+n}=W_i\), the chain \(W^{(2)}\) goes twice around \(E_{0.5}(\mathcal P)\). Therefore, 
\[
A^\ast(W^{(2)})
=
2A^\ast\left(E_{0.5}(\mathcal P)\right).
\]
By \eqref{eq:area-P-W-Q}, we get
\[
A^\ast(\mathcal Q)
=
A^\ast(\mathcal P)
-
2A^\ast\left(E_{0.5}(\mathcal P)\right).
\]

Since \(\mathcal P\) and \(\mathcal Q\) are positively oriented convex polygons, we have
\[
A^\ast(\mathcal P)=A(\mathcal P),
\qquad
A^\ast(\mathcal Q)=A(\mathcal Q).
\]
Using also
\[
A^\ast\left(E_{0.5}(\mathcal P)\right)\leqslant 0,
\]
we obtain
\begin{align}
\label{eq:Q-area}
A(\mathcal Q)
=
A(\mathcal P)
+
2\left|A^\ast\left(E_{0.5}(\mathcal P)\right)\right|.
\end{align}

By the classical isoperimetric inequality for \(2n\)-gons,
\begin{equation}
L(\mathcal Q)^2
\geqslant
8n\tan\left(\frac{\pi}{2n}\right)A(\mathcal Q),
\label{eq:classical-isoperimetric-Q}
\end{equation}
with equality if and only if \(\mathcal Q\) is a regular \(2n\)-gon. Using \eqref{eq:Q-perimeter} and \eqref{eq:Q-area} in \eqref{eq:classical-isoperimetric-Q}, we obtain
\[
L(\mathcal P)^2
\geqslant
8n\tan\left(\frac{\pi}{2n}\right)
\cdot\big(
A(\mathcal P)
+
2\left|A^\ast\left(E_{0.5}(\mathcal P)\right)\right|
\big),
\]
which proves the inequality.

The equality case follows from the equality case in \eqref{eq:classical-isoperimetric-Q}.
Equality holds if and only if $\mathcal{Q}$ is a regular $2n$-gon. Indeed, $\mathcal{Q}$ and $\mathcal P$ have the same side directions, while central symmetrization preserves widths. Hence, if $\mathcal{Q}$ is regular, then $\mathcal P$ is equiangular and has constant width. Conversely, if $\mathcal P$ is equiangular and has constant width, then $Q$ is regular.
\end{proof}

In Figure \ref{fig:pentagramOfCW} we illustrate an equiangular constant width $\cppos$ $\mathcal{P}$, for $n=5$, with the following ten vertices (coordinates are rounded to three decimal places):
$$
(26.349, 10.887),
(2.279, 28.374),
(-2.549, 28.374),
(-26.131, 11.241),$$ $$
(-27.680, 6.472),
(-18.522, -21.714),
(-14.860, -24.374),$$ $$
(14.590, -24.374),
(18.740, -21.359),
(27.783, 6.472).
$$

\begin{figure}[h!]
\centering
\includegraphics[width=0.50\linewidth]{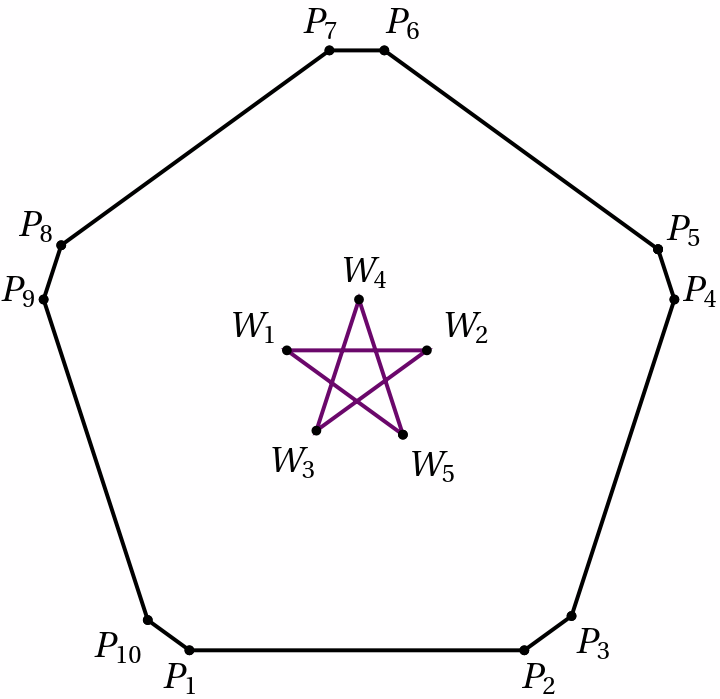}
\caption{Equiangular $\cppos$ of constant width and its Wigner caustic}
\label{fig:pentagramOfCW}
\end{figure}

\noindent By direct calculations, we get the following quantities, rounded to three decimal places:
\begin{align*}
    &L(\mathcal{P})\approx 171.391,\quad
    A(\mathcal{P})\approx 2137, \quad
    A^*\left(\Eq_{0.5}(\mathcal{P})\right)\approx -61.587,\\
    &L(\mathcal{P})^2-8\cdot5\tan\left(\frac{\pi}{2\cdot5}\right)\cdot\big(A(\mathcal{P})+2\left|A^*\left(\Eq_{0.5}(\mathcal{P}\right)\right|\big)\approx 0.
\end{align*}

\begin{cor}[Theorem 3.4 in \cite{Zwierz1}]\label{cor:IsoIneqOval}
    Let $\mathcal{O}$ be an oval. Then
    $$L(\mathcal{O})^2\geqslant 4\pi A(\mathcal{O})+8\pi\left| A^*\left(\Eq_{0.5}(\mathcal{O})\right)\right|,$$
    and the equality holds if and only if $\mathcal{O}$ is a curve of constant width.
\end{cor}

\begin{remark}
In the limiting case of Theorem~\ref{thm:discreteImpIsoIneq}, as the number of vertices tends to infinity, the above inequality yields the improved isoperimetric inequality for a smooth oval involving its Wigner caustic (see Corollary~\ref{cor:IsoIneqOval}). This result was originally proved by means of a Fourier-series approach -- see Theorem~3.4 in \cite{Zwierz1}.

It is worth noting that, in the special case of equiangular $\cppos$es, Theorem~\ref{thm:discreteImpIsoIneq} can also be proved by means of discrete Fourier series; see, for instance, the modern treatment in \cite{kwong}, the earlier approach developed in \cite{PerpPolygon}, and the references therein. We do not include this proof here, since the argument presented in the proof of Theorem \ref{thm:discreteImpIsoIneq} is simpler and applies in the more general setting, without assuming equiangularity. Nevertheless, the discrete Fourier-series point of view suggests that the strategy of first proving a discrete version of a result and then passing to the~smooth limit may be a fruitful and flexible approach. This indicates that such a~discrete-to-smooth passage may provide a useful framework for obtaining further results of this type.
\end{remark}

Now, we will find relations between oriented areas of the Wigner caustic, the~centre symmetry set, and the area of the initial $\cppos$.

\begin{thm}\label{thmMaxAreaWC}
Let \(\mathcal P\) be a $\cppos$ with $2n$ vertices.
Then
\begin{align*}
0\leqslant\left|A^*(\Eq_{0.5}(\mathcal P))\right|
<
\frac14 A(\mathcal P)
\quad\text{and}\quad
0\leqslant\left|A^*(\Eq_{0.5}(\mathcal P))\right|\leqslant
\left|A^*(\Css(\mathcal{P}))\right|<A(\mathcal{P}).
\end{align*}
Moreover, the equality cases in the above non-strict inequalities hold if and only if $\mathcal{P}$ is centrally symmetric. Furthermore, strict inequalities are sharp -- more precisely,
\[
\sup_{\mathcal P}
\frac{\left|A^*(\Eq_{0.5}(\mathcal P))\right|}
{A(\mathcal P)}
=
\frac14
\quad\text{and}\quad 
\sup_{\mathcal P}
\frac{\left|A^*(\Css(\mathcal P))\right|}
{A(\mathcal P)}
=
1,
\]
where the supremum is taken over all $\cppos$es with $2n$ vertices.
\end{thm}

\begin{proof}
Let \(K\subset\mathbb R^2\) be the convex body bounded by \(\mathcal P\). We use the previously established identity (see equation \eqref{eq:Q-area}):
\[
A^*(\Eq_{0.5}(\mathcal P))
=
\frac12
\left(
A(\mathcal P)
-
A\left(\frac{K-K}{2}\right)
\right)\leqslant 0.
\]

By the Rogers--Shephard inequality in the plane,
\begin{align}
\label{eq:RogersShephardIneq}
A\left(\frac{K-K}{2}\right)
\leqslant
\frac32A(\mathcal P),
\end{align}
with the equality case only for simplices. Therefore, in our case, the inequality in \eqref{eq:RogersShephardIneq} is strict.
Using the identity \eqref{eq:Q-area}, we get
\[
\left|A^*(\Eq_{0.5}(\mathcal P))\right|
=
\frac12
\left(
A\left(\frac{K-K}{2}\right)-A(\mathcal P)
\right)
<
\frac12
\left(
\frac32A(\mathcal P)-A(\mathcal P)
\right)=\frac{1}{4}A(\mathcal{P}),
\]
which ends the first part of the proof of the theorem.

Let us recall that
\[
D_i-P_i=\lambda_i d_i,
\]
where \(\lambda_i\in(0,1)\). Since
$d_i=P_{i+n}-P_i=-2Q_i,$
we have
\[
D_i
=
P_i+\lambda_i d_i
=
W_i+Q_i-2\lambda_i Q_i=W_i+(1-2\lambda_i)Q_i.
\]
For simplicity put 
$\eta_i=1-2\lambda_i.$
Then
$D_i=W_i+\eta_i Q_i$.
Consequently,
\[
D_i=P_{i+1}+\lambda_i d_{i+1}.
\]
Using
$P_{i+1}=W_{i+1}+Q_{i+1}$,
$d_{i+1}=-2Q_{i+1}$,
we also get
\[
D_i=W_{i+1}+\eta_i Q_{i+1}.
\]
Therefore, 
\begin{align}
\label{eq:ditwoformulas}
D_i=W_i+\eta_i Q_i
=
W_{i+1}+\eta_i Q_{i+1}.
\end{align}
In particular,
\begin{align}
\label{eq:wi1widiff}W_{i+1}-W_i
=
-\eta_i(Q_{i+1}-Q_i).
\end{align}

We now compute the difference between the oriented area of the centre symmetry set and the oriented area of the Wigner caustic.

First,
\[
A^*(\Eq_{0.5}(\mathcal P))
=
\frac12\sum_{i=1}^n[W_i,W_{i+1}]
=
\frac12\sum_{i=1}^n[W_i,W_{i+1}-W_i].
\]
Using \eqref{eq:wi1widiff},
we get
\begin{align}
\label{eq:astarWC} A^*(\Eq_{0.5}(\mathcal P))
=
-\frac12
\sum_{i=1}^n
\eta_i[W_i,Q_{i+1}-Q_i].
\end{align}

Now consider the centre symmetry set. We have
\[
A^*(\Css(\mathcal P))
=
\frac12\sum_{i=1}^n[D_i,D_{i+1}].
\]
Using
\eqref{eq:ditwoformulas}, we get
$D_{i+1}=W_{i+1}+\eta_{i+1}Q_{i+1}$
and
\[
D_{i+1}-D_i
=
(\eta_{i+1}-\eta_i)Q_{i+1}.
\]
Hence,
\begin{align*}
[D_i,D_{i+1}]
&=
[W_{i+1}+\eta_iQ_{i+1},(\eta_{i+1}-\eta_i)Q_{i+1}]
\\
&=
(\eta_{i+1}-\eta_i)[W_{i+1},Q_{i+1}]
+
\eta_i(\eta_{i+1}-\eta_i)[Q_{i+1},Q_{i+1}]\\
&=
(\eta_{i+1}-\eta_i)[W_{i+1},Q_{i+1}].
\end{align*}
Thus, 
\begin{align*}
A^*(\Css(\mathcal P))
=
\frac12
\sum_{i=1}^n
(\eta_{i+1}-\eta_i)[W_{i+1},Q_{i+1}]
=
\frac12
\sum_{i=1}^n
(\eta_i-\eta_{i-1})[W_i,Q_i].
\end{align*}

We now use summation by parts. Since the indices are cyclic,
\[
\sum_{i=1}^n(\eta_i-\eta_{i-1})[W_i,Q_i]
=
-\sum_{i=1}^n
\eta_i
\left(
[W_{i+1},Q_{i+1}]-[W_i,Q_i]
\right).
\]
Therefore, 
\begin{align}
\label{eq:cssareaformula1}
A^*(\Css(\mathcal P))
=
-\frac12
\sum_{i=1}^n
\eta_i
\left(
[W_{i+1},Q_{i+1}]-[W_i,Q_i]
\right).
\end{align}

Let us compute the expression in parentheses. From \eqref{eq:wi1widiff}
we get
\begin{align*}
[W_{i+1},Q_{i+1}]
&=
[W_i-\eta_i(Q_{i+1}-Q_i),Q_{i+1}]
=
[W_i,Q_{i+1}]
-
\eta_i[Q_{i+1}-Q_i,Q_{i+1}]\\
&=
[W_i,Q_{i+1}]
+
\eta_i[Q_i,Q_{i+1}].
\end{align*}
Thus, 
\[
[W_{i+1},Q_{i+1}]-[W_i,Q_i]
=
[W_i,Q_{i+1}-Q_i]
+
\eta_i[Q_i,Q_{i+1}].
\]
Substituting this into the formula for \(A^*(\Css(\mathcal P))\) (see \eqref{eq:cssareaformula1}), we obtain
\[
A^*(\Css(\mathcal P))
=
-\frac12
\sum_{i=1}^n
\eta_i[W_i,Q_{i+1}-Q_i]
-
\frac12
\sum_{i=1}^n
\eta_i^2[Q_i,Q_{i+1}].
\]
By the formula for the oriented area of the Wigner caustic (see equation \eqref{eq:astarWC}) we get that
\[
A^*(\Css(\mathcal P))
=
A^*(\Eq_{0.5}(\mathcal P))
-
\frac12
\sum_{i=1}^n
\eta_i^2[Q_i,Q_{i+1}].
\]
Since
$\eta_i=1-2\lambda_i$,
we finally obtain
\[
\left|A^*(\Css(\mathcal P))\right|
-
\left|A^*(\Eq_{0.5}(\mathcal P))\right|
=
\frac12
\sum_{i=1}^n
(1-2\lambda_i)^2[Q_i,Q_{i+1}]\geqslant 0,
\]
since the auxiliary polygon with vertices \(Q_i\) is positively oriented. Note that the equality in the above inequality holds if and only if $\lambda_i=0.5$ for all $i$ -- it means that $\mathcal{P}$ is centrally symmetric.

Finally, let us compare the absolute value of the oriented area of the centre symmetry set with the area of the original polygon.

Since \(0<\lambda_i<1\), we have
\[
\left|A^*(\Css(\mathcal P))\right|
-
\left|A^*(\Eq_{0.5}(\mathcal P))\right|
\leqslant
\frac12
\sum_{i=1}^n[Q_i,Q_{i+1}]=\frac14
\sum_{i=1}^{2n}[Q_i,Q_{i+1}]=\frac{1}{2}A(\mathcal{Q}),
\]
where \(\mathcal Q=\partial\left(\frac{K-K}{2}\right)\). 

On the other hand, from the area identity for the Wigner caustic (see equation~\eqref{eq:Q-area}),
\[
A^*(\Eq_{0.5}(\mathcal P))
=
\frac12\left(A(\mathcal P)-A(\mathcal Q)\right),
\]
we obtain
\[
\left|A^*(\Css(\mathcal P))\right|
\leqslant
\frac12\left(A(\mathcal Q)-A(\mathcal P)\right)
+
\frac12A(\mathcal Q)=
A(\mathcal Q)-\frac12A(\mathcal P).
\]

By the Rogers--Shephard inequality in the plane (see inequality \eqref{eq:RogersShephardIneq}), we get
\begin{align*}
\left|A^*(\Css(\mathcal P))\right|<A(\mathcal P).
\end{align*}

It remains to show that the constant \(\frac14\) and $1$ are sharp in the inequalities.

First, consider the case \(n=3\). Let $T$ be the following equilateral triangle of side length $\sqrt{3}$ with center in $(0,0)$:
\[
T=\operatorname{conv}\left\{\left(0,1\right),
\left(-\frac{\sqrt3}{2},-\frac12\right),
\left(\frac{\sqrt3}{2},-\frac12\right)\right\}.
\]
For \(0<\varepsilon<\frac{\sqrt{3}}{2}\), define the hexagon \(\mathcal P_\varepsilon\) by the vertices:
\begin{align*}
P_1&=\left(-\frac{\sqrt3}{2}+\varepsilon,-\frac12\right),&
P_2&=\left(\frac{\sqrt3}{2}-\varepsilon,-\frac12\right),&
P_3&=\left(\frac{\sqrt3}{2}-\frac{\varepsilon}{2},
-\frac12+\frac{\sqrt3\varepsilon}{2}\right),\\
P_4&=\left(\frac{\varepsilon}{2},
1-\frac{\sqrt3\varepsilon}{2}\right),&
P_5&=\left(-\frac{\varepsilon}{2},
1-\frac{\sqrt3\varepsilon}{2}\right),&
P_6&=\left(-\frac{\sqrt3}{2}+\frac{\varepsilon}{2},
-\frac12+\frac{\sqrt3\varepsilon}{2}\right).
\end{align*}
\begin{figure}[h!]
\centering
\includegraphics[width=0.45\linewidth]{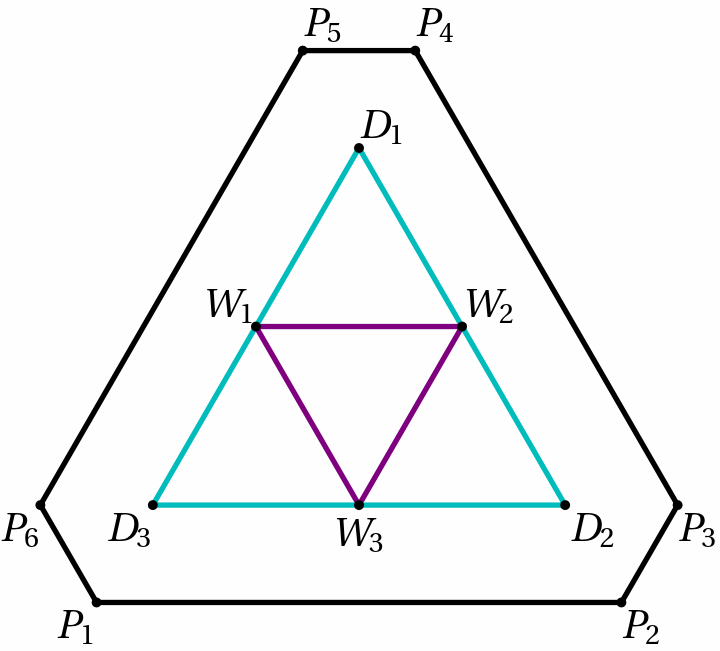}
\caption{Polygon $\mathcal{P}_\varepsilon$}
\label{fig:hexagonproof}
\end{figure}

This is triangle \(T\) with its three vertices cut off by segments parallel to the opposite sides (see Figure \ref{fig:hexagonproof}). Then, by direct calculations, one gets the following results for the Wigner caustic and the centre symmetry set of $\mathcal{P}_\varepsilon$: 
\[
\frac{
\left|A^*(\Eq_{0.5}(\mathcal P_\varepsilon))\right|}
{A(\mathcal P_\varepsilon)}
=
\frac{
\frac{3\sqrt{3}}{16}(1-\sqrt{3}\varepsilon)^2
}
{
\frac{3\sqrt{3}}{4}(1-\varepsilon^2)
}
=
\frac{(1-\sqrt{3}\varepsilon)^2}{4(1-\varepsilon^2)}
\text{ and }
\frac{
\left|A^*(\Css(\mathcal P_\varepsilon))\right|
}
{
A(\mathcal P_\varepsilon)
}
=\frac{(1-\sqrt{3}\varepsilon)^2}{1-\varepsilon^2}.
\]
Thus, the constants \(\frac{1}{4}\) and \(1\) are sharp for the Wigner caustic and the centre symmetry set in the case \(n=3\), respectively.

Now assume that $n>3$ and put
$r=n-3$.
We shall replace the two
opposite sides $P_2P_3$ and $P_5P_6$ by two polygonal arcs, each
consisting of $r+1$ edges.

Since the sides $P_2P_3$ and $P_5P_6$ are parallel and oppositely
oriented, there exists a unique number $\mu_\varepsilon>0$ such that
$P_6-P_5=-\mu_\varepsilon(P_3-P_2)$.

Define the negative homothety by the following equation
$$H_\varepsilon(x)
=
P_5-\mu_\varepsilon(x-P_2).$$
Then
$H_\varepsilon(P_2)=P_5$,
and
$H_\varepsilon(P_3)=P_6$.

For sufficiently small $\rho>0$, let $\gamma_\rho$ be a strictly
convex circular arc joining $P_2$ to $P_3$, lying on the exterior side
of the segment $P_2P_3$, and contained in \linebreak the~$\rho$-neighbourhood of
$P_2P_3$. Choose points
$Q_0^\rho,Q_1^\rho,\ldots,Q_{r+1}^\rho$
on $\gamma_\rho$, in their cyclic order, such that
$Q_0^\rho=P_2$,
$Q_{r+1}^\rho=P_3$ (see Figure \ref{fig:decagonproof}),
and, as $\rho\to 0$,
\[
Q_j^\rho
\to
Q_j^0
:=
\left(1-\frac{j}{r+1}\right)P_2
+
\frac{j}{r+1}P_3,
\qquad
j=0,\ldots,r+1.
\]

\begin{figure}[h!]
\centering
\includegraphics[width=0.444\linewidth]{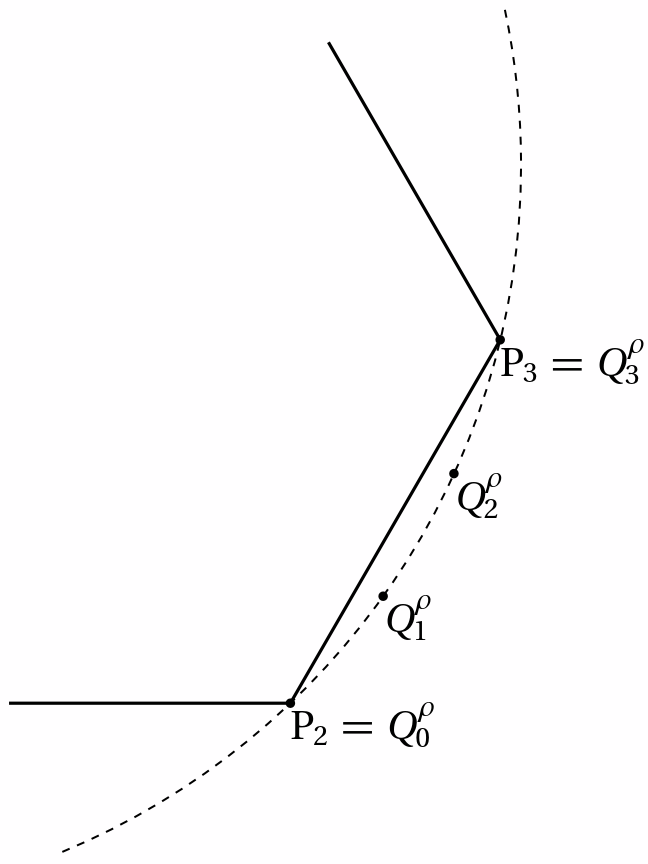}
\caption{A circular arc through $P_2$ and $P_3$ in $\mathcal{P}_{\varepsilon}$}
\label{fig:decagonproof}
\end{figure}

Define the corresponding points on the opposite arc by
$S_j^\rho
=
H_\varepsilon(Q_j^\rho)$ for 
$j=0,\ldots,r+1$.
Then
$S_0^\rho=P_5$,
$S_{r+1}^\rho=P_6$,
and
$S_{j+1}^\rho-S_j^\rho
=
-\mu_\varepsilon
\bigl(Q_{j+1}^\rho-Q_j^\rho\bigr)$ for
$j=0,\ldots,r$.
Therefore,
\[
Q_j^\rho Q_{j+1}^\rho
\parallel
S_j^\rho S_{j+1}^\rho,
\qquad
j=0,\ldots,r.
\]

Let $\mathcal P_{\varepsilon,\rho}^{(n)}$ be the polygon whose vertices,
in cyclic order, are
\[
P_1,\,
Q_0^\rho,Q_1^\rho,\ldots,Q_{r+1}^\rho,\,
P_4,\,
S_0^\rho,S_1^\rho,\ldots,S_{r+1}^\rho.
\]
It has
$1+(r+2)+1+(r+2)
=
2r+6
=
2n$
vertices and by the construction, its opposite sides are parallel. 
Moreover, for sufficiently small $\rho>0$, the directions of the
successive sides occur in the same cyclic order as the directions of
the sides of $\mathcal P_\varepsilon$. Hence,
$\mathcal P_{\varepsilon,\rho}^{(n)}$ is convex and therefore is a
$\cppos$. Furthermore, the two polygonal arcs converge to the original sides $P_2P_3$
and $P_5P_6$, respectively. Consequently,
$A\bigl(\mathcal P_{\varepsilon,\rho}^{(n)}\bigr)
\to
A(\mathcal P_\varepsilon)$,
$A^*\left(
\Eq_{0.5}\bigl(\mathcal P_{\varepsilon,\rho}^{(n)}\bigr)
\right)
\to
A^*\left(
\Eq_{0.5}(\mathcal P_\varepsilon)
\right)$,
and 
$A^*\left(
\Css\bigl(\mathcal P_{\varepsilon,\rho}^{(n)}\bigr)
\right)
\to
A^*\left(
\Css(\mathcal P_\varepsilon)
\right)$
as $\rho\to 0$.

Finally, let $\varepsilon_k\to0$. For each $k$, choose $\rho_k>0$
sufficiently small so that the two ratios corresponding to
$\mathcal P_{\varepsilon_k,\rho_k}^{(n)}$ differ from the corresponding
ratios for $\mathcal P_{\varepsilon_k}$ by less than $1/k$. Since
\[
\frac{
\left|
A^*\left(
\Eq_{0.5}(\mathcal P_{\varepsilon_k})
\right)
\right|
}{
A(\mathcal P_{\varepsilon_k})
}
\to
\frac14
\quad\text{and}\quad
\frac{
\left|
A^*\left(
\Css(\mathcal P_{\varepsilon_k})
\right)
\right|
}{
A(\mathcal P_{\varepsilon_k})
}
\to
1,
\]
we conclude that
\[
\frac{
\left|
A^*\left(
\Eq_{0.5}\bigl(\mathcal P_{\varepsilon_k,\rho_k}^{(n)}\bigr)
\right)
\right|
}{
A\bigl(\mathcal P_{\varepsilon_k,\rho_k}^{(n)}\bigr)
}
\to
\frac14
\quad\text{and}\quad
\frac{
\left|
A^*\left(
\Css\bigl(\mathcal P_{\varepsilon_k,\rho_k}^{(n)}\bigr)
\right)
\right|
}{
A\bigl(\mathcal P_{\varepsilon_k,\rho_k}^{(n)}\bigr)
}
\to
1.
\]
Hence, the constants $\frac14$ and $1$ are sharp for every $n\geqslant3$.
\end{proof}

As a corollary of Theorem \ref{thmMaxAreaWC}, we obtain the analogous result for smooth ovals and their Wigner caustics by approximating an oval by the converging sequence of $\cppos$es.

\begin{cor}[\cite{ZwierzMix}]
    Let $\mathcal{O}$ be an oval. Then
    $$\left|A^*(\Eq_{0.5}(\mathcal{O}))\right|\leqslant \frac{1}{4}A(\mathcal{O})\quad\text{and}\quad
0\leqslant\left|A^*(\Eq_{0.5}(\mathcal O))\right|\leqslant
\left|A^*(\Css(\mathcal{O}))\right|\leqslant A(\mathcal{O}).$$
\end{cor}

\section*{Acknowledgements}

\noindent The authors would like to thank Wojciech Domitrz for valuable discussions. 
They are also grateful to Marcos Craizer for introducing them to the topic of convex polygons with parallel opposite sides, and to Przemys\l aw G\'orka for posing the~question whether the~original set can be recovered from its associated singular sets. 
This question provided additional motivation for the present work.

\bibliographystyle{amsalpha}

\end{document}